\documentclass{amsart}
\date{\today}
\title{Simplicity of higher rank triplet $W$-algebras
}
\author{SHOMA SUGIMOTO}
\address{
Kyushu University, Faculty of Mathematics, Fukuoka 819-0386 JAPAN}
\email{sugimoto.shoma.657@m.kyushu-u.ac.jp}
\usepackage{amsmath,amssymb,amsthm,mathrsfs,amscd,ascmac,bm}
\usepackage[dvipdfmx]{hyperref}
\usepackage{color}
\theoremstyle{plain}
\newtheorem{thm}{Theorem}[section]
\newtheorem{prop}[thm]{Proposition}
\newtheorem{lemm}[thm]{Lemma}
\theoremstyle{definition}
\newtheorem{dfn}[thm]{Definition}
\theoremstyle{remark}
\newtheorem{rmk}[thm]{Remark}
\theoremstyle{Main Theorem}

\theoremstyle{Conjecture}

\theoremstyle{Corollary}
\newtheorem{cor}[thm]{Corollary}
\theoremstyle{Condition}

\newcommand{\Vmod}[1]{V_{#1}}
\newcommand{\xmod}[1]{\mathscr{V}_{#1}}
\newcommand{\wmod}[1]{\bm{W}(#1)}
\newcommand{\Wmod}[1]{W_{#1}}
\newcommand{\uniwalg}[1]{\bm{W}^{#1}(\mathfrak{g})}
\newcommand{\irrwalg}[1]{\bm{W}_{#1}(\mathfrak{g})}
\newcommand{\verma}[1]{M(#1)}
\newcommand{\irr}[1]{L(#1)}
\newcommand{\afverma}[1]{\hat{M}(#1)}
\newcommand{\afirr}[1]{\hat{L}(#1)}
\newcommand{\afweyl}[1]{\hat{V}(#1)}
\newcommand{\wverma}[1]{\bm{M}(#1)}
\newcommand{\wirr}[1]{\bm{L}(#1)}
\newcommand{\wfock}[1]{\bm{F}(#1)}
\newcommand{\warf}[3]{\bm{T}^{#1}_{#2, #3}}

\newcommand{\pal}{-\sqrt{p}\alpha+\lambda}

\newcommand{\pQl}{\sqrt{p}Q+\lambda}
\newcommand{\pPl}{\sqrt{p}P+\lambda}
\newcommand{\pQ}{\sqrt{p}Q}
\newcommand{\pP}{\sqrt{p}P}
\newcommand{\lz}{\lambda_0}
\newcommand{\lp}{\lambda_p}
\newcommand{\Lz}{\Lambda_0}
\newcommand{\Lp}{\Lambda_p}

\begin{document}
\maketitle
\begin{abstract}
We show that the higher rank triplet $W$-algebra $\Wmod{\pQ}$ is simple for $p\geq h-1$.
Furthermore, we show that the $\Wmod{\pQ}$-module $\Wmod{\sqrt{p}(Q-\lz)+\lp}$ introduced in \cite{FT} is simple if $\sqrt{p}\lp$ is in the closure of the fundamental alcove, and give the decomposition as a direct sum of simple $\irrwalg{p-h}$-modules.
\end{abstract}

\section{Introduction}\label{intro}
\markboth{SHOMA SUGIMOTO}{Simplicity of higher rank triplet $W$-algebras}
The {\it triplet $W$-algebra} (see, e.g., \cite{AM1}-\cite{AM3}, \cite{FGST1}-\cite{FGST3}, \cite{NT}, \cite{TW}) is one of the most well-known examples of $C_2$-cofinite and irrational vertex operator algebra \cite{FB,FHL,Kac}, and related to many interesting topics. It has been widely believed that the {\it higher rank triplet $W$-algebra} $\Wmod{\pQ}$ inherits properties of the triplet $W$-algebra such as simplicity, $C_2$-cofiniteness, irrationality, Kazhdan-Lusztig correspondence, etc. However, apart from the triplet $W$-algebra, very little is known about $\Wmod{\pQ}$. 
The main purpose of this paper is to show the simplicity of $\Wmod{\pQ}$.

Let $\mathfrak{g}$ be a finite-dimensional simply-laced simple Lie algebra and $h$ be the Coxeter number.
For a fixed integer $p\geq h-1$, let $\sqrt{p}Q$ be the rescaled root lattice of $\mathfrak{g}$.
We consider the simple lattice vertex operator algebra $\Vmod{\pQ}$ associated to $\pQ$,
and the irreducible $\Vmod{\pQ}$-module $\Vmod{\pQl}$ corresponding to $\lambda\in\Lambda=\frac{1}{\sqrt{p}}P/\sqrt{p}Q$ (see \cite{D}).
Each $\Vmod{\pQl}$ has a natural $B$-action, and thus, we obtain the homogeneous vector bundle $\xmod{\pQl}=G\times_B\Vmod{\pQl}$ over the flag variety $G/B$. We define the higher rank triplet $W$-algebra as $\Wmod{\pQ}=H^0(\xmod{\pQ})$ and $\Wmod{\pQ}$-modules $\Wmod{\pQl}=H^0(\xmod{\pQl})$ (see \cite{FT}).
In the previous article \cite{S} we have shown that
$\Wmod{\pQ}$ is isomorphic to the intersection of the kernels of {\it narrow screening operators} on $\Vmod{\pQ}$.
Denote by $\rho$ and $\theta$ the Weyl vector and the highest root of $\mathfrak{g}$, respectively.

\begin{thm}\label{mthm1}
The vertex operator algebra $\Wmod{\pQ}$ is simple for any
finite dimensional simply-laced simple Lie algebra $\mathfrak{g}$ and any integer $p\geq h-1$.
Moreover, 
$\Wmod{\pQl}$ is a simple $\Wmod{\pQ}$-module for $\lambda\in\Lambda$ such that $(\sqrt{p}\lp+\rho,\theta)\leq p$
(where $\lp$ is given by \eqref{defLambda}), 
\end{thm}

Theorem \ref{mthm1} has the following application to the representation theory of {\it affine $W$-algebras}.
Denote by $\irrwalg{p-h}$ the unique simple quotient of the affine $W$-algebra $\uniwalg{p-h}$ at level $p-h$ 
(\cite{FF}).
In \cite{S}, for each $\lambda\in\Lambda$, we gave a natural $G\times\irrwalg{p-h}$-module structure on $\Wmod{\pQl}$ and a $G\times\irrwalg{p-h}$-module isomorphism
\begin{align}\label{(0)}
\Wmod{\pQl}\simeq\bigoplus_{\alpha\in P_+\cap Q}L(\alpha+\lz)\otimes\wmod{\pal}\subseteq\Vmod{\pQl}.
\end{align}
Here $P_+$ is the set of dominant integral weights, 
$\lz\in P_+$ satisfies $\lambda=-\sqrt{p}\lz+\lp$ and $(\lz,\rho)=1$,
$L(\beta)$ is the simple $\mathfrak{g}$-module with highest weight $\beta$, and 
each $\wmod{\pal}$ is a certain $\irrwalg{p-h}$-submodule of $\Vmod{\pQl}$.
% which has the same character as $\warf{p}{\sqrt{p}\lp}{\alpha+\lz}$ in \cite{ArF} under the condition $(\sqrt{p}\lp+\rho,\theta)\leq p$.
%When $(\sqrt{p}\lp+\rho,\theta)\leq p$, $\wmod{\pal}$ has the same character as $\irrwalg{p-h}$-modules $\warf{p}{\sqrt{p}\lp}{\alpha+\lz}$ and $\warf{\frac{1}{p}}{\alpha+\lz}{\sqrt{p}\lp}$,
%=H^0_{\alpha+\lambda_0}(\hat{V}(\sqrt{p}\lambda_p+p\Lambda_0))
For the Weyl module $\hat{V}(\nu+\kappa\Lambda_0)$ induced from $L(\nu)$ and the twisted quantum Drinfeld-Sokolov reduction $H^0_{\mu}(?)$, denote by $\warf{\kappa+h}{\nu}{\mu}$ the $\irrwalg{\kappa}$-module $H^0_\mu(\hat{V}(\nu+\kappa\Lambda_0))$ (see \cite[Section 2]{ArF}).
Let  $\wirr{\gamma_{\mu}}$ be the simple $\irrwalg{p-h}$-module with higheast weight $\gamma_{\mu}$ (see \cite{Ar}).

\begin{thm}\label{mthm2}
Let $\mathfrak{g}$ be a finite dimensional simply-laced simple Lie algebra, $p\geq h-1$ and $\lambda\in\Lambda$ such that $(\sqrt{p}\lp+\rho,\theta)\leq p$. Then for any $\alpha\in P_+\cap Q$, we have
\begin{align}
\wmod{\pal}\simeq\wirr{\gamma_{\pal-p\rho}}\simeq\warf{p}{\sqrt{p}\lp}{\alpha+\lz}\simeq\warf{\frac{1}{p}}{\alpha+\lz}{\sqrt{p}\lp}
\nonumber
\end{align}
as $\irrwalg{p-h}$-modules (we used the Feigin-Frenkel duality $\irrwalg{p-h}\simeq\irrwalg{\frac{1}{p}-h}$ (\cite{FF1,FF2,ACL}) in the last isomorphism).
%$\wmod{\pal}$ is a simple $\irrwalg{p-h}$-module.
%that is isomorphic to $\warf{p}{\sqrt{p}\lp}{\alpha+\lz}$ in \cite{ArF}.
\end{thm}
In particular, Theorem \ref{mthm2} gives an extension of \cite[Theorem 2.2, 2.3]{ArF}.

%For $\lambda\in\Lambda$ such that $(\sqrt{p}\lp+\rho,\theta)\leq p$, $\wmod{\pal}$ has the same character as the $\irrwalg{p-h}$-module $T^p_{\sqrt{p}\lp,\alpha+\lz}=H^0_{\alpha+\lz}(\hat{V}(\sqrt{p}\lp+(p-h)\Lambda_0))$ in \cite{ArF}, where $H^0_{\mu}$ is the twisted quantum Drinfeld-Sokolov reduction and $\hat{V}(\nu+\kappa\Lambda_0)$ is the Weyl module induced from $L(\nu)$ (see \cite[Section 2]{ArF}). 

Let us explain the outline of the proof of Theorem \ref{mthm1} and Theorem \ref{mthm2} briefly.

First, we prove the isomorphism
\begin{align}\label{isom9/28}
\Wmod{\pQl}\simeq H^{l(w_0)}(\xmod{\pQl}^\ast(-2\rho))^\ast\simeq\Wmod{\pQ-w_0(\lambda)}^\ast,
\end{align}
where $w_0$ is the longest element in the Weyl group $W$
of $\mathfrak{g}$ and $M^\ast$ denotes the restricted dual of a $\Wmod{\pQ}$-module $M$.
The first isomorphism in \eqref{isom9/28} is the Serre duality, which holds for any $\lambda\in\Lambda$.
The second isomorphism
is derived using \cite[Theorem 4.8]{S} that holds under the condition  $(\sqrt{p}\lp+\rho,\theta)\leq p$. 

The isomorphism \eqref{isom9/28}
for $\lp=0$ implies that $\Wmod{\pQ}$ admits a non-degenerate $\Wmod{\pQ}$-invariant bilinear form
in the sense of \cite{FHL}.
Therefore, $\Wmod{\pQ}$ is simple.
In turn, Theorem \ref{mthm2} for $\lp=0$ follows from the quantum Galois theory (\cite{DLM,McR}).
In particular, for $\lp=0$, 
we find that the character of $\wmod{\pal}$ coincides with that of the corresponding simple $\irrwalg{p-h}$-module $\wirr{\gamma_{\pal-p\rho}}$ given in \cite{Ar}.
We can see that this coincidence of characters is also valid for $\lp$ satisfying the condition $(\sqrt{p}\lp+\rho,\theta)\leq p$, which proves Theorem \ref{mthm2}.
Theorem \ref{mthm1} follows from \eqref{isom9/28} and Theorem \ref{mthm2}.
\medskip

{\it Acknowledgements} \  \ The author wishes to express his gratitude to his supervisor Tomoyuki Arakawa for lots of advices and discussions to improve this paper. He thanks to Thomas Creutzig, Boris Feigin, Ryo Fujita, Naoki Genra, Shigenori Nakatsuka and Ryo Sato for useful comments and discussions. 
Finally, he appreciates the referees for the thoughtful and constructive feedback.
This work was supported by JSPS KAKENHI Grant number 19J21384.

\section{Preliminaries}\label{preliminaries}
\subsection{$\mathscr{O}_X$-vertex operator algebra and the Serre duality}
In this paper, for a ringed space, the structure sheaf is of $\mathbb{C}$-algebras.

\begin{dfn}\label{thefirstdefn}
Let $R$ be a $\mathbb{C}$-algebra and $V$ be an $R$-module.
We call $V$ a {\it vertex operator algebra over $R$} if $V$ satisfies the following conditions:
\begin{enumerate}
\item
there exists a set $\{V_n~|~n\in \mathbb{Z}\}$ of free $R$-modules 
of finite rank
such that
$V=\bigoplus_{n\in \mathbb{Z}}V_n$ and $V_{n}=0$ for $n\ll 0$,
\item
there exist a {\it vacuum vector} $|0\rangle\in V_0$ and a {\it conformal vector} $\omega\in V_2$, and an $R$-module homomorphism $Y_n\colon V\rightarrow\operatorname{End}_R(V)$, $a\mapsto a_{(n)}$ for each $n\in\mathbb{Z}$ such that
\begin{enumerate}
\item\label{lemm:new1.1}
for any $a\in V_{i}$, $n\in\mathbb{Z}$, $j\in \mathbb{Z}$, 
we have $a_{(n)}|_{V_j}\in\operatorname{Hom}_R(V_j, V_{j+i-n-1})$,
\item
we have $|0\rangle_{(n)}=\delta_{n,-1}\operatorname{id}_V$.
Also for any $a\in V$ and $n\geq -1$, we have $a_{(n)}|0\rangle=\delta_{n,-1}a$,
\item
for $n\in\mathbb{Z}$, let $L_n$ denote $\omega_{(n+1)}$. Then for any $m,n\in\mathbb{Z}$, we have
\begin{align*}
[L_m,L_n]=(m-n)L_{m+n}+\binom{m+1}{3}\delta_{m+n,0}c'_V\operatorname{id}_V,
\end{align*}
where $c'_V\in \mathbb{C}$.
Also 
for each $N\in\mathbb{Z}$, $V_N$ is a generalized eigenspace of $L_0$ with the eigenvalue $N$,
and for any $a\in V$, $N,n\in \mathbb{Z}$, we have $(L_{-1}a)_{(n)}=-na_{(n-1)}$.
\item\label{lemm:new1.2}
for any $N\in \mathbb{Z}$, $a,b\in V$, and $m,n,r\in\mathbb{Z}$, we have
\begin{align*}
&\sum_{i=0}^\infty\binom{r}{i}(a_{(m+i)}b)_{(r+n-i)}|_{V_N}\\
=&\sum_{i=0}^\infty\binom{m}{i}(-1)^i(a_{(m+r-i)}b_{(n+i)}-(-1)^mb_{(m+n-i)}a_{(r+i)})|_{V_N}.
\end{align*}
Note that both hands of above equation are finite sum,
\end{enumerate}
\end{enumerate}
We also define a {\it $V$-module} and {\it $V$-module homomorphism} in obvious ways.
\end{dfn}

\begin{rmk}\label{rmkCalgebra}
In this paper, we will mainly deal with the case where $R=\mathcal{O}_{G/B}(U)$ for some open subset $U$ of a flag variety $G/B$ over $\mathbb{C}$.
\end{rmk}

\begin{rmk}\label{rmkeigen}
In this paper, the case where $L_0$ acts semisimply only appears.
\end{rmk}

For $\mathbb{C}$-algebras $R,R'$, a $\mathbb{C}$-algebra homomorphism $f\colon R\rightarrow R'$, and a vertex operator algebra $V$ over $R$, the $R'$-module $R'\otimes_RV$ has a structure of a vertex operator algebra over $R'$ in an obvious way. 
When $R$ is obvious, we write $\otimes$ instead of $\otimes_R$ for the tensor product over $R$.
For a vertex operator algebra $V$ over $R$ and a $V$-module $M$, denote $V=\bigoplus_{n\in\mathbb{Z}}V_n$ and $M=\bigoplus_{\Delta\in\mathbb{C}}M_\Delta$ by the conformal grading of $V$ and $M$, respectively.

Let $(X,\mathscr{O}_X)$ be a ringed space. 
For open subsets $U_1$, $U_2$ of $X$ such that $U_1\subseteq U_2$, an $\mathscr{O}_X$-module $\mathcal{M}$ and $s\in\mathscr{M}(U_2)$, denote by $r^{\mathscr{M}}_{U_1,U_2}$ (or simply $r_{U_1,U_2}$)  the restriction map from $\mathscr{M}(U_2)$ to $\mathscr{M}(U_1)$, and $s|_{U_1}$ the image $r^{\mathscr{M}}_{U_1,U_2}(s)$ of $s$.
 Let us recall that for $\mathscr{O}_X$-modules $\mathscr{M}$ and $\mathscr{N}$, a morphism $\phi\colon\mathscr{M}\rightarrow\mathscr{N}$ consists of $\mathscr{O}_X(U)$-module homomorphisms $\phi(U)\colon\mathscr{M}(U)\rightarrow\mathscr{N}(U)$ for each open subset $U$ of $X$, such that for any open subset $U_1$ of $U_2$, we have $r^{\mathscr{N}}_{U_1,U_2}\circ\phi(U_2)=\phi(U_1)\circ r^{\mathscr{M}}_{U_1,U_2}$. 
 For an $\mathscr{O}_X$-module $\mathscr{F}$ and an open subset $U$ of $X$, we write $\mathscr{F}|_U$ for the restriction of $\mathscr{F}$ to $U$.
 Also, denote by $\operatorname{Hom}_{\mathscr{O}_X}(\mathscr{M}, \mathscr{N})$ the $\mathbb{C}$-module consisting of all morphisms from $\mathscr{M}$ to $\mathscr{N}$.
 Then the $\mathscr{O}_X$-module $\mathscr{H}{\it om}_{\mathscr{O}_X}(\mathscr{M},\mathscr{N})$, so-called sheaf Hom, is defined by
$U\mapsto\operatorname{Hom}_{{\mathscr{O}_X}|_{U}}(\mathscr{M}|_{U}, \mathscr{N}|_{U})$.

For a $\mathbb{C}$-algebra $R$ and a ringed space $(X,\mathscr{O}_X)$, denote $\operatorname{Mod}_R$ and  $\operatorname{Mod}_{\mathscr{O}_X}$ by the categories of $R$-modules and $\mathscr{O}_X$-modules, respectively.
We regard $\operatorname{Mod}_{\mathscr{O}_X(X)}$ as a subcategory of $\operatorname{Mod}_{\mathscr{O}_X}$ by the embedding $M\mapsto\mathscr{O}_X\otimes_{\mathscr{O}_X(X)}M$.
Note that these categories are (co)complete abelian.

\begin{dfn}\label{dfn:OX vertex operator algebra}
For a ringed space $(X,\mathscr{O}_X)$, an $\mathscr{O}_X$-module $\mathscr{V}$ is called an {\it $\mathscr{O}_X$-vertex operator algebra} if 
\begin{enumerate}
\item
there exists a set $\{\mathscr{V}_n~|~n\in\mathbb{Z}\}$ of locally free of finite rank $\mathscr{O}_X$-modules such that 
$\mathscr{V}=\bigoplus_{n\in\mathbb{Z}}\mathscr{V}_n$ as $\mathscr{O}_X$-modules and $\mathscr{V}_n=0$ for $n\ll 0$,
\item for any open subset $U\subseteq X$, $\mathscr{V}(U)$ is a vertex operator algebra over $\mathscr{O}_X(U)$ and $\mathscr{V}_n(U)=\mathscr{V}(U)_n$ for any $n\in\mathbb{Z}$,
\item for any open subsets $U_1\subseteq U_2\subseteq X$, the restriction map $r_{U_1,U_2}\colon\mathscr{V}(U_2)\rightarrow\mathscr{V}(U_1)$ 
defines a  
vertex operator algebra homomorphism.
\end{enumerate}
\end{dfn}

\begin{dfn}\label{dfn:OX VOAmodule}
For a ringed space $(X,\mathscr{O}_X)$ and an $\mathscr{O}_X$-vertex operator algebra $\mathscr{V}$, an $\mathscr{O}_X$-module $\mathscr{M}$ is called a {\it $\mathscr{V}$-module} if 
\begin{enumerate}
\item
there exists a set $\{\mathscr{M}_\Delta~|~\Delta\in\mathbb{C}\}$ of locally free of finite rank $\mathscr{O}_X$-modules such that 
$\mathscr{M}=\bigoplus_{\Delta\in\mathbb{C}}\mathscr{M}_\Delta$ as $\mathscr{O}_X$-modules and $\mathscr{M}_{\Delta-n}=0$ for any $\Delta\in\mathbb{C}$ and $n\gg 0$,
\item 
for any open subset $U\subset X$, $\mathscr{M}(U)$ is a $\mathscr{V}(U)$-module,
\item 
for any open subsets $U_1\subseteq U_2\subseteq X$, the restriction map $r^{\mathscr{M}}_{U_1,U_2}\colon\mathscr{M}(U_2)\rightarrow\mathscr{M}(U_1)$ defines the linear map such that for any $a\in\mathscr{V}(U_2)$ and $n\in\mathbb{Z}$, $r^{\mathscr{M}}_{U_1,U_2}\circ a_{(n)}=r_{U_1,U_2}(a)_{(n)}\circ r^{\mathscr{M}}_{U_1,U_2}$. 
\end{enumerate}
By abuse of notation, when $\mathscr{M}$ is an $\mathscr{O}_X\otimes_{\mathscr{O}_X(X)}V$-module for some vertex operator algebra $V$ over $\mathscr{O}_X(X)$, we simply call $\mathscr{M}$ a {\it $V$-module}.
\end{dfn}

\begin{lemm}\label{lemm:H^0(V)-module}
Let $(X,\mathscr{O}_X)$ be a ringed space, $\mathscr{V}$ be an $\mathscr{O}_X$-vertex operator algebra, and $\mathscr{M}$ be a $\mathscr{V}$-module.
Then $\mathscr{M}$ has a structure of $H^0(\mathscr{V})$-module.
\end{lemm}
\begin{proof}
For any open subset $U$ of $X$, $s\in H^0(\mathscr{V})$, $f\in\mathscr{O}_X(U)$, and $n\in\mathbb{Z}$, set $(f\otimes s)_{(n)}^{\mathscr{V}(U)}=f(s|_U)_{(n)}$. Then the action satisfies the axioms in Definition \ref{dfn:OX VOAmodule}.
\end{proof}

\begin{dfn}\label{dfn:dagger}
Let $V$ be a vertex operator algebra over a $\mathbb{C}$-algebra $R$.
Then for $a=\sum_{\Delta\in\mathbb{Z}}a_\Delta\in V$, $a_\Delta\in V_\Delta$ and $n\in\mathbb{Z}$, 
the operator $(a_{(n)}^\dagger)^M\in\operatorname{End}_{R}(M)$ is defined by
\begin{align*}
(a_{(n)}^\dagger)^M=\sum_{\Delta\in\mathbb{Z}}(-1)^{\Delta}\sum_{m\geq 0}(\frac{L_1^m}{m!}a_\Delta)^M_{(-n-m+2\Delta-2)}.
\end{align*}
\end{dfn}

\begin{lemm}\label{additivefunctor}
Let $(X,\mathscr{O}_X)$ be a ringed space, $\mathscr{V}$ be an $\mathscr{O}_X$-vertex operator algebra, $\mathscr{M}$ be a $\mathscr{V}$-module, and $F$ be a $\mathbb{C}$-additive functor from a (co)complete abelian subcategory $\mathcal{C}$ of $\operatorname{Mod}_{\mathscr{O}_X}$ to $\operatorname{Mod}_{\mathscr{O}_X}$ such that if $\mathscr{N}$ is locally free of finite-rank, then so is $F(\mathscr{N})$. 
\begin{enumerate}
\item
If $F$ is covariant, then $F(\mathscr{M})$ has the $H^0(\mathscr{V})$-module structure as follows: 
for an open subset $U\subseteq X$, $a\in H^0(\mathscr{V})$, $n\in\mathbb{Z}$, and $\Delta\in\mathbb{C}$,
we have $F(\mathscr{M})_\Delta:=F(\mathscr{M}_\Delta)$, $a_{(n)}^{F(\mathscr{M})(U)}:=F(a_{(n)}^{\mathscr{M}(U)})$.
\item
If $F$ is contravariant, then $F(\mathscr{M})$ has the $H^0(\mathscr{V})$-module structure as follows:
for an open subset $U\subseteq X$, $a\in H^0(\mathscr{V})$, $n\in\mathbb{Z}$, and $\Delta\in\mathbb{C}$,
we have $F(\mathscr{M})_\Delta:=F(\mathscr{M}_\Delta)$, $a_{(n)}^{F(\mathscr{M})(U)}:=F((a_{(n)}^\dagger)^{\mathscr{M}(U)})$.
\end{enumerate}
\end{lemm}
\begin{proof}
It is easily checked that the axioms in Definition \ref{dfn:OX VOAmodule} are satisfied.
\end{proof}

\begin{cor}\label{cor:dualmoduledef}
Let $V$ be a vertex operator algebra over a $\mathbb{C}$-algebra $R$, and let $M$ be a $V$-module.
Then 
\begin{align*}
M^{\ast}=\bigoplus_{\Delta\in\mathbb{C}}M^\ast_{\Delta}=\bigoplus_{\Delta\in\mathbb{C}}\operatorname{Hom}_R(M_\Delta,R)
\end{align*}
has the $V$-module structure defined by
\begin{align*}
(a_{(n)}\phi)(v)=\phi(a_{(n)}^\dagger v),
\end{align*}
where $a\in V$, $\phi\in M^\ast$, $v\in M$, and $a_{(n)}^\dagger$ is given in Definition \ref{dfn:dagger}.
We call $M^\ast$ the {\it restricted dual $V$-module} of $M$ (see \cite{FHL}).
\end{cor}
\begin{proof}
Let us consider the case where $X=\operatorname{Spec}R$, $\mathcal{V}=\mathscr{O}_X\otimes V$, $\mathscr{M}=\mathscr{O}_X\otimes M$, $\mathcal{C}=\operatorname{Mod}_R$, and the contravariant functor $F\colon\mathcal{C}\rightarrow\operatorname{Mod}_{\mathscr{O}_X}$ is defined by
\begin{align}
F(M)=M^\ast,~F(f)=?\circ f
\end{align}
for any $M\in\operatorname{Mod}_R$ and $R$-module homomorphism $f$.
Then by Lemma \ref{additivefunctor}, the assertion is proved.
\end{proof}

\begin{cor}\label{mathcaldual}
Let $(X,\mathscr{O}_X)$ be a ringed space, $\mathscr{V}$ be an $\mathscr{O}_X$-vertex operator algebra, $\mathscr{M}$ be a $\mathscr{V}$-module.
Then 
\begin{align}
\mathscr{M}^\ast:=\bigoplus_{\Delta\in\mathbb{C}}\mathscr{M}^\ast_{\Delta}:=\bigoplus_{\Delta\in\mathbb{C}}\operatorname{\mathscr{H}{\it om}}_{\mathscr{O}_X}(\mathscr{M}_{\Delta},\mathscr{O}_X)
\end{align}
has the $\mathscr{V}$-module structure defined by
\begin{align}
s_{(n)}\phi(U)=\{\phi(U_0)\circ r^{\mathscr{M}}_{U_0,U}(s)_{(n)}^\dagger~|~\text{$U_0$ is an open subset of $U$.}\},
\end{align}
where $U$ is an open subset of $X$, $s\in\mathscr{V}(U)$, $n\in\mathbb{Z}$ and $\phi\in\mathscr{M}^\ast(U)$.
We call $\mathscr{M}^\ast$ the {\it restricted dual} of $\mathscr{M}$. 
\end{cor}
\begin{proof}
It is easy to check the axioms in Definition \ref{dfn:OX VOAmodule}.
\end{proof}

\begin{thm}\label{thm:serreduality}
Let $X$ be a projective Cohen-Macaulay scheme of pure dimension $N$ over $\mathbb{C}$, $\mathscr{V}$ be an $\mathscr{O}_X$-vertex operator algebra, and $\mathscr{M}$ be a $\mathscr{V}$-module. Then the Serre duality (see, e.g. \cite[III, (7.7)]{H})
\begin{align}\label{serreeq}
H^n(\mathscr{M})=\bigoplus_{\Delta\in\mathbb{C}}H^n(\mathscr{M}_\Delta)\simeq\bigoplus_{\Delta\in\mathbb{C}}H^{N-n}(\mathscr{M}_\Delta^{\ast}\otimes\omega_X)^\ast=H^{N-n}(\mathscr{M}^{\ast}\otimes\omega_X)^\ast
\end{align}
provides an $H^0(\mathscr{V})$-module isomorphism, where $\omega_X$ is the dualizing sheaf of $X$.
\end{thm}

\begin{proof}
By Lemma \ref{additivefunctor} and Corollary \ref{mathcaldual}, $\mathscr{M}^\ast$ and $\mathscr{M}^\ast\otimes\omega_X$ are $H^0(\mathscr{V})$-modules. Moreover, because right derived functors are $\mathbb{C}$-additive (see \cite[III, (1.1 A)]{H}), by Lemma \ref{additivefunctor}, $H^n(\mathscr{M})$, $H^{N-n}(\mathscr{M}^{\ast}\otimes\omega_X)^\ast$, $\operatorname{Ext}^n(\mathscr{M}^\ast\otimes\omega_X,\omega_X)$ and $\operatorname{Ext}^n(\mathscr{O}_X,\mathscr{M})$ have the $H^0(\mathscr{V})$-module structures.
Because \eqref{serreeq} is defined by the isomorphisms \cite[III, (6.3), (6.7), (7.6)]{H},
it is enough to show that they are $H^0(\mathscr{V})$-module isomorphisms.
Since the isomorphisms \cite[III, (6.3), (7.6)]{H} are natural, namely, defined by the natural isomorphisms between the functors, they are $H^0(\mathscr{V})$-module isomorphisms. 

Let us show that \cite[III, (6.7)]{H} is the $H^0(\mathscr{V})$-module isomorphism.
For $n\geq 0$, $\phi_n$ denotes the natural isomorphism between the functors $\operatorname{Ext}^n(\mathscr{M}^\ast\otimes\omega_X,?)$ and $\operatorname{Ext}^n(\mathscr{O}_X,\mathscr{M}\otimes\omega_X^\ast\otimes?)$ in \cite[III, (6.7)]{H}.
For any $s\in H^0(\mathscr{V})$ and $m\in\mathbb{Z}$, the $\mathscr{O}_X$-module endomorphism $s_{(m)}$ on $\mathscr{M}$ induces the natural transformations from $\operatorname{Hom}_{\mathscr{O}_X}(\mathscr{M}^\ast\otimes\omega_X,?)$ and $\operatorname{Hom}_{\mathscr{O}_X}(\mathscr{O}_X,\mathscr{M}\otimes\omega_X^\ast\otimes?)$ to themselves, respectively. We use the same letter $s_{(m)}$ for these natural transformations. Then, by \cite[II, Ex.5.1]{H}, we have $s_{(m)}\circ\phi_0=\phi_0\circ s_{(m)}$.
By the universality of $\operatorname{Ext}^n(\mathscr{M}^\ast\otimes\omega_X,?)\simeq\operatorname{Ext}^n(\mathscr{O}_X,\mathscr{M}\otimes\omega_X^\ast\otimes?)$, we have $s_{(m)}\circ\phi_n=\phi_n\circ s_{(m)}$, where by abuse of notations, these $s_{(m)}$ mean the natural transformations from $\operatorname{Ext}^n(\mathscr{M}^\ast\otimes\omega_X,?)$ and $\operatorname{Ext}^n(\mathscr{O}_X,\mathscr{M}\otimes\omega_X^\ast\otimes?)$ to themselves, respectively, that are uniquely determined by the universality.
Thus, \cite[III, (6.7)]{H} gives the $H^0(\mathscr{V})$-module isomorphism
\begin{align}
\operatorname{Ext}^n(\mathscr{M}^\ast\otimes\omega_X,\omega_X)\simeq\operatorname{Ext}^n(\mathscr{O}_X,\mathscr{M}\otimes\omega_X^\ast\otimes\omega_X)=\operatorname{Ext}^n(\mathscr{O}_X,\mathscr{M})
\end{align}
for any $n\geq 0$.
\end{proof}

\begin{rmk}
Let us give \eqref{serreeq} more explicitly in the case of $n=0$.
We have the natural perfect pairing (see \cite[III, (7.6)]{H})
\begin{align}\label{(616)}
\operatorname{Hom}_{\mathscr{O}_{X}}(\mathscr{M}^\ast\otimes\omega_X,\omega_X)\times H^{N}(\mathscr{M}^\ast\otimes\omega_X)\rightarrow H^{N}(\omega_X)\simeq\mathbb{C}
\end{align}
and the linear isomorphism
\begin{align}\label{618}
\phi\colon H^0(\mathscr{M})\simeq\operatorname{Hom}_{\mathscr{O}_X}(\mathscr{M}^\ast\otimes\omega_X,\omega_X),~s\mapsto\phi(s),
\end{align}
where for any open subset $U\subseteq X$, $f\in\mathscr{M}^\ast(U)$ and $x\in\omega_X(U)$, $\phi(s)$ is defined by 
\begin{align}
\phi(s)(U)(f\otimes x):=f(U)(s|_U)x.
\end{align}
By Lemma \ref{additivefunctor}, we give the $H^0(\mathscr{V})$-module structure on $\operatorname{Hom}_{\mathscr{O}_X}(\mathscr{M}^\ast\otimes\omega_X,\omega_X)$ as follows: for any open subset $U\subseteq X$, $a\in H^0(\mathscr{V})$ and $n\in\mathbb{Z}$, 
\begin{align}
(a_{(n)}\phi(s))(U):=\phi(s)(U)\circ a_{(n)}^\dagger|_U.
\end{align}
Then \eqref{618} is the $H^0(\mathscr{V})$-module isomorphism because
\begin{align}
\phi(a_{(n)}s)(U)(f\otimes x)=f(U)(a_{(n)}s|_U)x=(a_{(n)}^\dagger f)(U)(s|_U)x=(a_{(n)}\phi(s))(U)(f\otimes x).\nonumber
\end{align}
Thus, we have 
\begin{align}\label{(620.1)}
R^{N}\Gamma(\phi(a_{(n)}s))=R^{N}\Gamma(\phi(s))\circ a_{(n)}^\dagger=a_{(n)}R^{N}\Gamma(\phi(s)),
\end{align}
where the $a_{(n)}^\dagger$ in the second term is the $H^0(\mathscr{V})$-action on $H^{N}(\mathscr{M})$, and the $a_{(n)}$ in the third term is that on $H^{N}(\mathscr{M})^\ast$.
Thus, by \eqref{(616)} and \eqref{(620.1)}, the linear isomorphism
\begin{align}\label{(621)}
H^0(\mathscr{M})\simeq H^{N}(\mathscr{M}^\ast\otimes\omega_X)^\ast,~s\mapsto R^{N}\Gamma(\phi(s))
\end{align}
is also the $H^0(\mathscr{V})$-module isomorphism.
\end{rmk}

\subsection{Our setting}\label{sect:oursetting}
Unless otherwise noted, the ground field is the complex number field $\mathbb{C}$ below.
Let $\mathfrak{g}$ be a simply-laced simple Lie algebra of rank $l$, and $\mathfrak{g}=\mathfrak{n}_-\oplus\mathfrak{h}\oplus\mathfrak{n}_+$ the triangular decomposition, $\mathfrak{h}$ the Cartan subalgebra, $\mathfrak{b}=\mathfrak{n}_-\oplus\mathfrak{h}$ the Borel subalgebra, $G$ and $B$ the semisimple, simply-connected, complex algebraic groups corresponding to $\mathfrak{g}$ and $\mathfrak{b}$, respectively. We adopt the standard numbering for the simple roots $\{\alpha_1,\ldots,\alpha_l\}$ of $\mathfrak{g}$ as in \cite{B} and $\{\omega_1,\dots,\omega_l\}$ denotes the corresponding fundamental weights.
Denote by $\Delta^{\pm}$ the sets of positive roots and negative roots of $\mathfrak{g}$, respectively.
Let $Q$ be the root lattice of $\mathfrak{g}$, $P$ the weight lattice of $\mathfrak{g}$, $P_+$ the set of dominant integral weights of $\mathfrak{g}$. Denote by $(\cdot , \cdot)$ the invariant form of $\mathfrak{g}$ normalized as $(\alpha_i,\alpha_i)=2$ for any $1\leq i\leq l$, $W$ the Weyl group of $\mathfrak{g}$ generated by the simple reflections $\{\sigma_i\}_{i=1}^l$, $(c^{ij})$ the inverse matrix to the Cartan matrix of $\mathfrak{g}$, $\rho$ the half sum of positive roots, $\theta$ the highest root of $\mathfrak{g}$, $h$ the Coxeter number of $\mathfrak{g}$. 
We choose the set $\Lz$ of representatives of generators of the abelian group $P/Q$ in $P_+$ as \cite{Kac2}, that is, $\lz\in\Lz\subseteq P_+$ satisfies $(\lz,\theta)=1$. 
For $\mu\in\mathfrak{h}^\ast$ and the longest element $w_0$ in $W$, we use the notation $\mu'$ for $-w_0(\mu)$.
For $\sigma\in W$, $l(\sigma)$ denotes the length of $\sigma$.
Let $h_1,\cdots,h_l$ be the basis of $\mathfrak{h}$ corresponding to the simple roots by $(\cdot , \cdot)$.
Denote by $\irr{\beta}$ the simple $\mathfrak{g}$-module with highest weight $\beta$.
For $\mu\in P$, $\mathbb{C}(\mu)$ denotes the one-dimensional $\mathfrak{b}$-module such that the action of $\mathfrak{n}_-$ is trivial and $h_i$ acts as $(\alpha_i,\mu)\operatorname{id}$ for $1\leq i\leq l$.
For a $B$-module $V$ and $\mu\in P$, we use the letter $V(\mu)$ for the $B$-module $V\otimes\mathbb{C}(\mu)$.
For a vertex operator algebra $V$, $\mathcal{U}(V)$ denotes the universal enveloping algebra of $V$.
For a $V$-module $M$ and $m\in M$, $\mathcal{U}(V)m$ means the $V$-submodule of $M$ generated by $m$.

We fix an integer $p\in\mathbb{Z}_{\geq 2}$.
Set
\begin{align}
\Lp=\{\sum_{i=1}^l\frac{s_i}{\sqrt{p}}\omega_i~|~0\leq s_i\leq p-1\}\subseteq\frac{1}{\sqrt{p}}P_+.
\end{align}
For $\mu\in\frac{1}{\sqrt{p}}P$, denote by $\mu_0$ and $\mu_p$ the elements of $P$ and $\Lambda_p$, respectively, such that $\mu=-\sqrt{p}\mu_0+\mu_p$.
Let 
\begin{align}
\Vmod{\pQ}=\bigoplus_{\alpha\in Q}\wfock{-\sqrt{p}\alpha}
\end{align}
be the simple lattice vertex operator algebra associated to the $\pQ$.
Here, for $\mu\in\mathfrak{h}^\ast$, $\wfock{\mu}$ denotes the Fock space over the rank $l$ Heisenberg vertex operator algebra $\wfock{0}$ with highest weight $\mu$.

\begin{rmk}\label{rmk2022.4.24}
Our notation $\wfock{\mu}$ corresponds to $\pi_{\kappa+h,\sqrt{\kappa+h}\mu}$ in \cite[Section 3]{ACL}, where $\kappa$ is a complex number such that $\kappa\not=-h$.
Namely, we have $\wfock{\mu}\simeq \pi_{\kappa+h,\sqrt{\kappa+h}\mu}$ as $\wfock{0}$-modules, where the $\wfock{0}$-module structure on $\pi_{\kappa+h,\sqrt{\kappa+h}\mu}$ is given by $(at^{-1}|0\rangle)_{(n)}^{\pi_{\kappa+h,\sqrt{\kappa+h}\mu}}=\frac{1}{\sqrt{\kappa+h}}at^n$ for $a\in\mathfrak{h}$, $n\in\mathbb{Z}$, and the vacuum vector $|0\rangle$ of $\wfock{0}$.
In this paper, we mainly consider the cases where $\kappa+h=p$ or $\frac{1}{p}$.
%In particular, we have
%\begin{align}\label{align1,2022.4.25}
%\wfock{\mu}_{\operatorname{top}}\simeq(\pi_{p,\sqrt{p}\mu})_{\operatorname{top}}\simeq\mathbb{C}_{{\gamma_{-\sqrt{p}\mu-p\rho}}}
%\end{align}
%as $\operatorname{Zhu}(\mathcal{W}_{p-h}(\mathfrak{g})$-modules, where $\gamma_{\lambda}$ is given in \cite[(56)]{Ar} and the third isomorphism in \eqref{align1,2022.4.25} follows from the proof of \cite[Proposition 6.2]{ACL} and the footnote in \cite[p.22]{ACL}.
\end{rmk}

We choose the conformal vector $\omega$ of $\Vmod{\pQ}$ as
 \begin{align}\label{1}
 \omega=\frac{1}{2}\sum_{1\leq i,j \leq l}c^{ij}(\alpha_{i})_{(-1)}\alpha_{j}+Q_{0}(\rho)_{(-2)}|0\rangle\in\wfock{0}\subseteq\Vmod{\pQ},
 \end{align}
 where $Q_{0}=\sqrt{p}-\frac{1}{\sqrt{p}}$.
 The central charge $c$ of $\omega$ is given by
 \begin{align}\label{2}
 c=l-12|Q_0\rho|^2=l-Q_0^2h\dim{\mathfrak{g}}.
 \end{align}
The parameter set $\Lambda$ of simple $\Vmod{\pQ}$-modules is given by
\begin{align}\label{defLambda}
\Lambda=\{\lambda=-\sqrt{p}\lz+\lp~|~\lz\in\Lz,~\lp\in\Lp\}.
\end{align}
By \cite{D}, the set of $\Vmod{\pQ}$-modules 
\begin{align}\label{(2.1)}
\{\Vmod{\pQl}=\bigoplus_{\alpha\in Q}\wfock{-\sqrt{p}\alpha+\lambda}=\bigoplus_{\alpha\in Q}\wfock{-\sqrt{p}(\alpha+\lz)+\lp}\}_{\lambda\in\Lambda}
\end{align}
provides a complete set of isomorphism classes of simple $\Vmod{\pQ}$-modules.

We also define the generalized lattice vertex operator algebra and their modules 
\begin{align}\label{3.1}
\Vmod{\pP}=\bigoplus_{\lz\in\Lz}\Vmod{\sqrt{p}(Q-\lz)},~\Vmod{\pP+\lp}=\bigoplus_{\lz\in\Lz}\Vmod{\sqrt{p}(Q-\lz)+\lp}
\end{align}
in the same manner as $\Vmod{\pQ}$ and $\Vmod{\pQl}$ (e.g., see \cite{BK,DL,McR}). 

For $\alpha\in Q$ and $\lambda\in\Lambda$, $|-\sqrt{p}\alpha+\lambda\rangle$ denotes the lattice point vector in $\Vmod{\pQl}$, so that
 $\wfock{-\sqrt{p}\alpha+\lambda}=\mathcal{U}(\wfock{0})|-\sqrt{p}\alpha+\lambda\rangle$. %Then 
For $\mu\in\frac{1}{\sqrt{p}}P$, the conformal weight $\Delta_{\mu}$ of $|\mu\rangle$ is given by
 \begin{align}\label{4}
 \Delta_{\mu}=\frac{1}{2}|\mu -Q_{0}\rho|^2 +\frac{c-l}{24}=\frac{1}{2}|\mu|^2 -Q_{0}(\mu, \rho)\in\mathbb{Q}.
 \end{align}
 In particular, for any $\lambda\in\Lambda$, we have $\Vmod{\pQl}=\bigoplus_{\Delta\in\mathbb{Q}}(\Vmod{\pQl})_\Delta$.
 
 \begin{dfn}\label{weylaction}
For $\sigma\in W$ and $\mu=-\sqrt{p}\mu_0+\mu_p\in\frac{1}{\sqrt{p}}P$, set the following $W$-action on $\frac{1}{\sqrt{p}}P$:
\begin{align}\label{(16.2)}
\sigma\ast\mu=-\sqrt{p}\mu_0+\frac{1}{\sqrt{p}}\sigma\circ\sqrt{p}\mu_p\in \frac{1}{\sqrt{p}}P,
\end{align}
where $\sigma\circ\mu=\sigma(\mu+\rho)-\rho$
 for $\mu\in\mathfrak{h}^\ast$.
For $\lp\in\Lambda$ and $\sigma\in W$, set
\begin{align}\label{(666)}
\epsilon_{\lp}(\sigma)=\frac{1}{\sqrt{p}}(\sigma\ast\lp-(\sigma\ast\lp)_p)\in P.
\end{align}
\end{dfn}

\begin{lemm}[{\cite[Corollary A.4]{S}}]\label{lemm11/13}
For $\lp\in\Lp$ such that $(\sqrt{p}\lp+\rho,\theta)\leq p$, we have $\epsilon_{\lp}(w_0)=-\rho$.
\end{lemm}

\begin{lemm}[{\cite[Lemma 3.7]{S}}]\label{condequiv}$ $
Let 
$w_0=\sigma_{i_{l(w_0)}}\dots\sigma_{i_1}$ be a minimal length expression of the longest element $w_0$ of $W$.
Then for $\lp\in\Lp$, the following conditions are equivalent:
\begin{enumerate}
\item\label{condition1}
for any $1\leq n\leq l(w_0)-1$, we have 
$(\epsilon_{\lp}(\sigma_{i_n}\dots\sigma_{i_1}),\alpha_{i_{n+1}})=0$,
\item\label{condition2}
we have $(\sqrt{p}\lp+\rho,\theta)\leq p$.
\end{enumerate}
\end{lemm}

For $1\leq i\leq l$, $\alpha\in P$, $\lambda\in\Lambda$ such that $0\leq s_i:=(\sqrt{p}\lp,\alpha_i)\leq p-2$, the {\it narrow screening operator} $S_{i,\lp}\in\operatorname{Hom}_{\mathbb{C}}(\wfock{-\sqrt{p}\alpha+\lambda},\wfock{-\sqrt{p}\alpha+\sigma_i\ast\lambda})$ is defined by
\begin{align}
S_{i,\lp}&=\int_{[\Gamma_{s_i+1}]}S_i(z_1)\dots S_i(z_{s_i+1})dz_1\dots dz_{s_i+1},
\end{align}
where $S_{i}(z)=|-\frac{1}{\sqrt{p}}\alpha_i\rangle(z)$ (see \cite{CRW,NT,S} for the precise definition of $S_{i}(z)$),
and the cycle $[\Gamma_{s_i+1}]$ is given in \cite[Proposition 2.1]{NT}. By \cite[Theorem 2.7, 2.8]{NT}, we have $S_{i,\lp}\not=0$. For convenience, we set $S_{i,\lp}=0$ for $\lambda\in\Lambda$ such that $s_i=p-1$.

\subsection{$\xmod{\pQ}$ and $\xmod{\pQl}(\mu)$}
For $1\leq i \leq l$, $\lambda\in\Lambda$, $\mu\in P$, we consider the following operators
\begin{align}\label{(21.5)}
f_i=&|\sqrt{p}\alpha_i\rangle_{(0)},\\
h_{i,\lp}(\mu)=&-\frac{1}{\sqrt{p}}(\alpha_i)_{(0)}+\frac{1}{\sqrt{p}}(\alpha_i , \lp+\sqrt{p}\mu)\operatorname{id}
\end{align}
acting on $\Vmod{\pQl}\otimes\mathbb{C}(\mu)$ ($f_i$ is called the {\it screening operator}).
When $\mu=0$, we use the notation $h_{i,\lp}$ instead of $h_{i,\lp}(0)$.
For any $\lambda\in\Lambda$, $\mu\in P$ and $1\leq i\leq l$, the operators $f_i$ and $h_{i,\lp}(\mu)$ act on $\Vmod{\pQl}\otimes\mathbb{C}(\mu)$ as differentials: namely, we have
\begin{align}
f_i(a_{(n)}v)&=(f_ia)_{(n)}v+a_{(n)}f_iv,\label{fidiff}\\
h_{i,\lp}(\mu)(a_{(n)}v)&=(h_{i,0}a)_{(n)}v+a_{(n)}h_{i,\lp}(\mu)v\label{hidiff}
\end{align}
for $a\in\Vmod{\pQ}$, $v\in\Vmod{\pQl}\otimes\mathbb{C}(\mu)$ and $n\in\mathbb{Z}$. 
It is straightforward to show that 
\begin{align}\label{(1024)}
[f_i,L_n]=[h_{i,\lp}(\mu),L_n]=[S_{i,\lp},L_n]=0,~[f_i,S_{j,\lp}]=0
\end{align}
for $1\leq i,j\leq l$ and $n\in\mathbb{Z}$, $\lambda\in\Lambda$ (see \cite{S}).
In particular, all $f_i$, $h_{i,\lp}(\mu)$ and $S_{i,\lp}$ preserve the conformal grading.

\begin{lemm}[{\cite{FT,S}}]\label{ft41}$ $
The operators $\{f_i,h_{i,\lp}\}_{i=1}^l$ give rise to an integrable action of $\mathfrak{b}$ on $\Vmod{\pQl}$.  
More generally, the operators $\{f_i,h_{i,\lp}(\mu)\}_{i=1}^l$ give rise to an integrable action of $\mathfrak{b}$ on $\Vmod{\pQl}\otimes\mathbb{C}(\mu)$ that preserves the conformal grading.  
\end{lemm}

By Lemma \ref{ft41}, we obtain the $B$-module $\Vmod{\pQl}(\mu)$.
For $\lambda\in\Lambda$ and $\mu\in P$, let
\begin{align}\label{(21.6)}
\rho^{\Vmod{\pQl}(\mu)}\colon B\rightarrow\operatorname{GL}_{\mathbb{C}}(\Vmod{\pQl}(\mu))
\end{align}
denote the action of $B$ on $V_{\sqrt{p}Q+\lambda}(\mu)$ given by Lemma \ref{ft41}. 
By \eqref{fidiff} and \eqref{hidiff}, the $B$-action $\rho^{\Vmod{\pQl}(\mu)}$ defines the automorphisms of $\Vmod{\pQl}(\mu)$ as $\Vmod{\pQ}$-module.
Namely, for $b\in B$, we have
\begin{align}\label{eqB}
\rho^{\Vmod{\pQl}(\mu)}(b)(a_{(n)}v)=(\rho^{\Vmod{\pQ}}(b)a)_{(n)}\rho^{\Vmod{\pQl}(\mu)}(b)v.
\end{align}

\begin{lemm}\label{isomirrV}
For any $\lambda\in\Lambda$, we have $\Vmod{\pQl}^\ast\simeq\Vmod{\pQ+w_0\ast\lambda'}$ as $\Vmod{\pQ}$-modules and $B$-modules.
\end{lemm}
\begin{proof}
For any $\lambda\in\Lambda$, since $\Vmod{\pQl}$ is a 
simple $\Vmod{\pQ}$-module, by \cite[Proposition 5.3.2]{FHL}, $\Vmod{\pQl}^\ast$ is isomorphic to a simple $\Vmod{\pQ}$-module $\Vmod{\pQ+\lambda^\ast}$ for some $\lambda^\ast\in\Lambda$. We have
\begin{align}\label{22.10}
w_0\ast\lambda'&=-\sqrt{p}\lz'+\frac{1}{\sqrt{p}}(w_0\circ\sqrt{p}\lp')=-\sqrt{p}(\lz'+\rho)+\frac{p-2}{\sqrt{p}}\rho-\lp,
\end{align}
and thus, 
\begin{align}\label{22.11}
\wfock{\sqrt{p}(\lz+\rho)+\frac{p-2}{\sqrt{p}}\rho-\lp}\subseteq\Vmod{\pQ+w_0\ast\lambda'}.
\end{align}
On the other hand, since 
\begin{align}\label{(607)}
(\alpha_i)_{(0)}^\dagger=-(\alpha_i)_{(0)}+2Q_0\operatorname{id}
\end{align}
for any $1\leq i\leq l$, by Corollary \ref{cor:dualmoduledef}, we have
\begin{align}\label{(608)}
(\alpha_i)_{(0)}|\lambda\rangle^\ast&=(\sqrt{p}(\lz+\rho)+\frac{p-2}{\sqrt{p}}\rho-\lp,\alpha_i)|\lambda\rangle^\ast.
\end{align}
Thus, by combining \eqref{22.11} with \eqref{(608)}, we have $\lambda^\ast=w_0\ast\lambda'$. By definition of the $B$-action $\rho^{\Vmod{\pQl}(\mu)}$, the remaining claim is also proved.
\end{proof}

For $\lambda\in\Lambda$ and $\mu\in P$, we consider the homogeneous vector bundle
\begin{align}\label{eqbdle}
\xmod{\pQl}(\mu)=G\times_B\Vmod{\pQl}(\mu)
\end{align}
over $G/B$, where the action of $B$ on $G$ is given by the right multiplication and that on $\Vmod{\pQl}(\mu)$ is given by $\rho^{\Vmod{\pQl}(\mu)}$. 
The total space $G\times_B\Vmod{\pQl}(\mu)$ is given by the quotient space $G\times\Vmod{\pQl}(\mu)/\sim$, where
\begin{align}
(g,v)\sim (gb^{-1}, \rho^{\Vmod{\pQl}(\mu)}(b)v)
\end{align}
for $g\in G$, $v\in\Vmod{\pQl}(\mu)$ and $b\in B$. 
By abuse of notations, we use the same notation $\xmod{\pQl}(\mu)$ for the sheaf of sections of \eqref{eqbdle}. 
When $\mu=0$, we use the letter $\xmod{\pQl}$ instead of $\xmod{\pQl}(0)$, for short. 
Clearly, for the line bundle $\mathscr{O}(\mu)=G\times_B\mathbb{C}(\mu)$ on $G/B$, we have $\xmod{\pQl}(\mu)=\xmod{\pQl}\otimes\mathscr{O}(\mu)$. 
By \eqref{(1024)}, $\xmod{\pQl}(\mu)$ is decomposed into the direct sum of locally free of finite rank $\mathscr{O}_{G/B}$-modules as 
\begin{align}\label{conformaldecomp}
\xmod{\pQl}(\mu)=\bigoplus_{\Delta\in\mathbb{Q}}\xmod{\pQl}(\mu)_\Delta=\bigoplus_{\Delta\in\mathbb{Q}}G\times_B(\Vmod{\pQl}(\mu))_\Delta.
\end{align} 
Moreover, since the sheaf of sections of $G\times_B\mathbb{C}|0\rangle$ is $\mathscr{O}_{G/B}$ and $\mathbb{C}|0\rangle$ is in the center of $\Vmod{\pQ}$, for any $s\in H^0(\xmod{\pQ})$ and $n\in\mathbb{Z}$, $s_{(n)}$ defines an $\mathscr{O}_{G/B}$-module homomorphism.
Denote by $\pi'$ for the projection from $G$ to $G/B$. 
For any open subset $U\subseteq G/B$, we give the vertex operator algebra $\mathscr{O}_G(\pi'^{-1}(U))\otimes V_{\sqrt{p}Q}$ over $\mathscr{O}_G(\pi'^{-1}(U))$ defined by 
$(f_1\otimes v_1)_{(n)}(f_2\otimes v_2)=f_1f_2\otimes {v_1}_{(n)}{v_2}$.
Then by \eqref{eqB}, the orbifold $(\mathscr{O}_G(\pi'^{-1}(U))\otimes\Vmod{\pQ})^B$ of the $B$-action defined by
\begin{align}
b(f(g)\otimes v)=f(gb^{-1})\otimes\rho^{\Vmod{\pQ}}(b)v
\end{align}
on $\mathscr{O}_G(\pi'^{-1}(U))\otimes\Vmod{\pQ}$
 is the vertex operator algebra over $(\mathscr{O}_G(\pi'^{-1}(U)))^B=\mathscr{O}_{G/B}(U)$.
 In the same manner, $(\mathscr{O}_G(\pi'^{-1}(U))\otimes\Vmod{\pQl}(\mu))^B$ is the $(\mathscr{O}_G(\pi'^{-1}(U))\otimes V_{\sqrt{p}Q})^B$-module for any $\lambda\in\Lambda$ and $\mu\in P$.
Since 
\begin{align}
\xmod{\pQl}(\mu)(U)=(\mathscr{O}_G(\pi'^{-1}(U))\otimes \Vmod{\pQl}(\mu))^B
\end{align}
and the compatibility between the action of $\xmod{\pQ}$ and the restriction maps of $\xmod{\pQ}$ and $\xmod{\pQl}(\mu)$ is clear, we obtain the following (here, let us recall Definition \ref{dfn:OX vertex operator algebra}).

\begin{lemm}\label{bdlemod}
The sheaf of sections $\xmod{\pQ}$ is an $\mathscr{O}_{G/B}$-vertex operator algebra. Moreover,
for $\lambda\in\Lambda$ and $\mu\in P$, $\xmod{\pQl}(\mu)$ is a $\xmod{\pQ}$-module. 
\end{lemm}

Let us define the main object in the present paper.
\begin{dfn}
By Lemma \ref{bdlemod}, the sheaf cohomology $H^0(\xmod{\pQ})$ is a vertex operator algebra, and we call $H^0(\xmod{\pQ})$ the {\it higher rank triplet $W$-algebra} associated to $\sqrt{p}Q$.
\end{dfn}

\begin{cor}\label{generalstructure}
For $0\leq i\leq l(w_0)$, $H^i(\xmod{\pQl}(\mu))$ has a natural $H^0(\xmod{\pQ})$-module and $G$-module structure.
Moreover, the $G$-action on $H^i(\xmod{\pQl}(\mu))$ gives the automorphisms of $H^0(\xmod{\pQ})$-module.
\end{cor}
\begin{proof}
For each $0\leq i\leq l(w_0)$, we consider the right derived functor $R^i\Gamma$ of the global section functor $\Gamma$ from $\operatorname{Mod}_{\mathscr{O}_{G/B}}$ to $\operatorname{Mod}_{\mathbb{C}}=\operatorname{Mod}_{\mathscr{O}_{G/B}(G/B)}\subseteq \operatorname{Mod}_{\mathscr{O}_{G/B}}$.
By Lemma \ref{lemm:H^0(V)-module} and Lemma \ref{bdlemod}, $\xmod{\pQl}(\mu)$ is an $H^0(\xmod{\pQ})$-module.
Hence by Lemma \ref{additivefunctor} and \eqref{conformaldecomp}, $H^i(\xmod{\pQl}(\mu))$ is the $H^0(\xmod{\pQ})$-module and 
 the set of morphisms
\begin{align}\label{rightderivedfunctor}
\{R^i\Gamma(s_{(n)})\}_{s\in H^0(\xmod{\pQ}),n\in\mathbb{Z}}
\end{align}
defines the $H^0(\xmod{\pQ})$-module structure on $H^i(\xmod{\pQl}(\mu))$.
For $g\in G$, we have the isomorphism of sheaves $g_\ast\xmod{\pQl}(\mu)\simeq\xmod{\pQl}(\mu)$ by
\begin{align}
g\colon g_\ast\xmod{\pQl}(\mu)(U)\rightarrow\xmod{\pQl}(\mu)(U),~t\mapsto gt
\end{align}
for any open subset $U\subseteq G/B$, where $g_\ast\xmod{\pQl}(\mu)(U):=\xmod{\pQl}(\mu)(g^{-1}U)$ and for $x\in U$, $(gt)(x):=t(g^{-1}x)$. 
Then we obtain the $G$-action on $H^i(\xmod{\pQl}(\mu))$ by $R^i\Gamma(g)$.
On the other hands, by definition, we have
\begin{align}
(gs)_{(n)}g=gs_{(n)}\colon g_\ast\xmod{\pQl}(\mu)\rightarrow\xmod{\pQl}(\mu).
\end{align}
By sending these morphisms by $R^i\Gamma$, the remaining claim is proved.
\end{proof}

By \eqref{(1024)}, for any subset $I'\subseteq\{1,\dots,l\}$, $\bigcap_{i\in I'}\ker f_i|_{\wfock{0}}$ and $\bigcap_{i\in I'}\ker S_{i,0}|_{\Vmod{\pQ}}$ are vertex operator full subalgebras of $\wfock{0}$ and $\Vmod{\pQ}$, respectively (namely, they have the conformal vector $\omega$ in \eqref{1}).

\begin{dfn}\label{defofwmodpal}
For $p\in\mathbb{Z}_{\geq2}$, $\lambda\in\Lambda$ and $\alpha\in P_+\cap Q$, set 
\begin{align}
\wmod{\pal}:=\bigcap_{i=1}^l(\mathcal{U}(\ker f_i|_{\wfock{0}})|-\sqrt{p}\alpha+\lambda\rangle)\subseteq\wfock{-\sqrt{p}\alpha+\lambda}.
\end{align}
Then 
\begin{align}
\wmod{0}=\bigcap_{i=1}^l\ker f_i|_{\wfock{0}}
\end{align}
is a vertex operator full subalgebra of $\wfock{0}$, and $\wmod{\pal}$ are $\wmod{0}$-modules.
\end{dfn}

\begin{lemm}[{\cite[Main Theorem (2), Theorem 4.14, Lemma 4.19]{S}}]%\label{condequiv}
$ $
For any $p\in\mathbb{Z}_{\geq 2}$ and $\lambda\in\Lambda$, we have an embedding
\begin{align}\label{embedding}
H^0(\xmod{\pQl})\hookrightarrow\bigcap_{i=1}^l\ker S_{i,\lambda}|_{\Vmod{\pQl}},~s\mapsto s(\operatorname{id}B).
\end{align}
When $\lambda=0$, \eqref{embedding} gives the vertex operator algebra isomorphism.
Moreover, for $\lambda\in\Lambda$ such that $(\sqrt{p}\lp+\rho,\theta)\leq p$, \eqref{embedding} gives the $H^0(\mathscr{V}_{\pQ})$-module isomorphism.
\end{lemm}

\begin{dfn}
For $\lambda\in\Lambda$, $\Wmod{\pQl}$ denotes the image of $H^0(\xmod{\pQl})$ by \eqref{embedding}.
\end{dfn}

\begin{thm}[\cite{S}]
For $\lambda\in\Lambda$, we have a $G\times\wmod{0}$-module isomorphism
\begin{align}
H^0(\xmod{\pQl})\simeq\Wmod{\pQl}\simeq\bigoplus_{\alpha\in{P_+\cap Q}}\irr{\alpha+\lz}\otimes\wmod{\pal}.
\end{align}
\end{thm}

\begin{rmk}\label{rmk:lowestconfweight}
The lowest conformal weight of $\Wmod{\pQl}$ is $\Delta_\lambda$ in \eqref{4}, and we have $(\Wmod{\pQl})_{\Delta_\lambda}=\irr{\lz}\otimes\mathbb{C}|\lambda\rangle$.
\end{rmk}

For $\lambda\in\Lambda$, denote by $\xmod{\pQl}^\ast$ the sheaf of sections of the dual bundle $G\times_B\Vmod{\pQl}^\ast$ over $G/B$, where the $B$-action on $\Vmod{\pQl}^\ast$ is given by the contragredient representation $(\rho^{\Vmod{\pQl}})^\ast$ of $\rho^{\Vmod{\pQl}}$.
Then for $1\leq i\leq l$, we have
\begin{align}
(\rho^{\Vmod{\pQl}})^\ast(h_i)&=\rho^{\pQ+w_0\ast\lambda'}(h_i)+\operatorname{id},\label{(611.1)}\\(\rho^{\Vmod{\pQl}})^\ast(f_i)&=\rho^{\pQ+w_0\ast\lambda'}(f_i).\label{(611.2)}
\end{align}
Here we identify $\Vmod{\pQl}^\ast$ with $\Vmod{\pQ+w_0\ast\lambda'}$ by Lemma \ref{isomirrV}, and \eqref{(611.1)} and \eqref{(611.2)} follow from \eqref{(607)} and  
$|\sqrt{p}\alpha_i\rangle_{(0)}^\dagger=-|\sqrt{p}\alpha_i\rangle_{(0)}$, respectively.
 Thus, we obtain the following.
 
\begin{lemm}\label{bdleisom}
For any $\lambda\in\Lambda$, we have $\xmod{\pQl}^\ast(-2\rho)\simeq\xmod{\pQ+w_0\ast\lambda'}(-\rho)$ as $\xmod{\pQ}$-modules. 
Moreover, for $0\leq i\leq l(w_0)$, we have 
\begin{align*}
H^i(\xmod{\pQl}^\ast(-2\rho))\simeq H^i(\xmod{\pQ+w_0\ast\lambda'}(-\rho))
\end{align*}
as $H^0(\xmod{\pQ})$-modules.
\end{lemm}

\section{Proofs of Theorem \ref{mthm1} and Theorem \ref{mthm2}}\label{equivsimple}
\subsection{Proof of \eqref{isom9/28}}
By Lemma \ref{thm:serreduality} and Lemma \ref{bdlemod}, we obtain the following.
\begin{prop}\label{serreduality}
For any $\lambda\in\Lambda$, we have 
$H^0(\xmod{\pQl})\simeq H^{l(w_0)}(\xmod{\pQl}^\ast(-2\rho))^\ast$
as $\Wmod{\pQ}$-modules.
\end{prop}
For $1\leq i\leq l$, let $P_i$ be the minimal parabolic subgroup of $G$ corresponding to $\alpha_i$. 
For $\lambda\in\Lambda$ and $\mu\in P$, we consider the homogeneous vector bundle $P_i\times_B\Vmod{\pQl}(\mu)$ over $P_i/B\simeq\mathbb{CP}^1$, and the sheaf of sections (by abuse of notations, we use the same letter).
In the same manner above, $P_i\times_B\Vmod{\pQ}$ is an $\mathscr{O}_{P_i/B}$-vertex operator algebra, and $P_i\times_B\Vmod{\pQl}(\mu)$ is a $P_i\times_B\Vmod{\pQ}$-module.
In \cite[Lemma 4.10]{S}, we proved that $\ker S_{i,0}|_{\Vmod{\pQ}}\simeq H^0(P_{i}\times_B\Vmod{\pQ})$ as vertex operator algebras.
We induce the $\ker S_{i,0}|_{\Vmod{\pQ}}$-module and $P_i$-module structure (and thus, $\Wmod{\pQ}$-module structure) on $H^n(P_i\times_B\Vmod{\pQl}(\mu))$ in the same manner as Corollary \ref{generalstructure}. 

By Lemma \ref{additivefunctor}, for $\lambda\in\Lambda$, $\mu\in P$, $p_1,p_2\in\mathbb{Z}$ and the projection $\pi_{i}'$ from $G/B$ to $G/P_{i}$, we induce the $\Wmod{\pQ}$-module structure on 
\begin{align}
(R^{p_1}\Gamma\circ R^{p_2}{{\pi_{i}'}}_{\ast})(\xmod{\pQl}(\mu))=H^{p_1}(G\times_{P_{i}}H^{p_2}(P_{i}\times_B\Vmod{\pQl}(\mu))).
\end{align}

\begin{prop}\label{nonserreduality}
For $\lambda\in\Lambda$ such that $(\sqrt{p}\lp+\rho,\theta)\leq p$, we have
$H^0(\xmod{\pQl})\simeq H^{l(w_0)}(\xmod{\pQ+\lambda'}^\ast(-2\rho))$ as $\Wmod{\pQ}$-modules.
\end{prop}

\begin{proof}
Let us take a minimal expression $w_0=\sigma_{i_{l(w_0)}}\cdots\sigma_{1}$ of $w_0$.
In \cite[Theorem 4.8]{S}, when $\lambda\in\Lambda$ satisfies the condition Lemma \ref{condequiv} \eqref{condition1}, we proved that for any $\sigma=\sigma_{i_{n-1}}\cdots\sigma_{i_1}$, $2\leq n\leq l(w_0)$, we have the $\Wmod{\pQ}$-module and $P_{i_n}$-module isomorphism
\begin{align}
H^1(P_{i_n}\times_B\Vmod{\pQ+\sigma\ast\lambda}(\epsilon_{\lp}(\sigma)))
&\simeq H^0(P_{i_n}\times_B\Vmod{\pQ+\sigma_{i_n}\sigma\ast\lambda}(\epsilon_{\lp}(\sigma_{i_n}\sigma)))\simeq 0,\label{(622.9)}\\
H^0(P_{i_n}\times_B\Vmod{\pQ+\sigma\ast\lambda}(\epsilon_{\lp}(\sigma)))&\simeq H^1(P_{i_n}\times_B\Vmod{\pQ+\sigma_{i_n}\sigma\ast\lambda}(\epsilon_{\lp}(\sigma_{i_n}\sigma)))\not\simeq 0.\label{(623)}
\end{align}
Then we obtain the $\Wmod{\pQ}$-module and $G$-module isomorphism
\begin{align}\label{(624)}
&\ H^{l(\sigma)}(\xmod{\pQ+\sigma\ast\lambda}(\epsilon_{\lp}(\sigma)))\\
\simeq&\ H^{l(\sigma)}(G\times_{P_{i_n}}H^0(P_{i_n}\times_BV_{\sqrt{p}Q+\sigma\ast\lambda}(\epsilon_{\lp}(\sigma))))\nonumber\\
\simeq&\ H^{l(\sigma)}(G\times_{P_{i_n}}H^1(P_{i_n}\times_BV_{\sqrt{p}Q+\sigma_{i_n}\sigma\ast\lambda}(\epsilon_{\lp}(\sigma_{i_n}\sigma))))\nonumber\\
\simeq&\ H^{l(\sigma)+1}(\xmod{\pQ+\sigma_{i_n}\sigma\ast\lambda}(\epsilon_{\lp}(\sigma_{i_n}\sigma))),\nonumber
\end{align}
where the first and third isomorphisms in \eqref{(624)} also follow from \eqref{(622.9)} and \eqref{(623)}, namely, the fact that the Leray spectral sequences have the trivial filtrations (see \cite[Theorem 4.8]{S}). By continuing \eqref{(624)} for $1\leq n\leq l(w_0)$ and using Lemma \ref{lemm11/13} and Lemma \ref{bdleisom} lastly, we obtain that
\begin{align}
H^0(\xmod{\pQl})\simeq\cdots\simeq H^{l(w_0)}(\xmod{\pQ+w_0\ast\lambda}(-\rho))\simeq H^{l(w_0)}(\xmod{\pQ+\lambda'}^\ast(-2\rho))
\end{align}
as $\Wmod{\pQ}$-modules and $G$-modules. 
\end{proof}

\begin{cor}\label{cor:thm1.2}
 For $\lambda\in\Lambda$ such that $(\sqrt{p}\lp+\rho,\theta)\leq p$, we have a $\Wmod{\pQ}$-module isomorphism
 \begin{align}\label{align:thm1.2}
 \Wmod{\pQl}\simeq\Wmod{\pQ+\lambda'}^\ast.
 \end{align}
\end{cor}
\begin{proof}
By \cite[Proposition 5.3.1]{FHL}, we have 
\begin{align}
H^{l(w_0)}(\xmod{\pQ+\lambda'}^\ast(-2\rho))\simeq (H^{l(w_0)}(\xmod{\pQ+\lambda'}^\ast(-2\rho))^\ast)^\ast
\end{align}
 as $\Wmod{\pQ}$-modules.
Thus, by Proposition \ref{serreduality} and Proposition \ref{nonserreduality}, we have 
\begin{align*}
\Wmod{\pQl}\simeq H^{l(w_0)}(\xmod{\pQ+\lambda'}^\ast(-2\rho))\simeq (H^{l(w_0)}(\xmod{\pQ+\lambda'}^\ast(-2\rho))^\ast)^\ast\simeq \Wmod{\pQ+\lambda'}^\ast.
\end{align*}
\end{proof}

\begin{rmk}
When $\mathfrak{g}=\mathfrak{sl}_2$, $\Wmod{\pQl}$ is self-dual for any $\lambda\in\Lambda$ (see \cite[Proposition 3.12]{NT}).
In fact, all $\lambda\in\Lambda$ satisfy $(\sqrt{p}\lp+\rho,\theta)\leq p$ and $\lambda=\lambda'$.
\end{rmk}

If $\Wmod{\pQ+\lambda'}^\ast\simeq\Wmod{\pQl}$ as $\Wmod{\pQ}$-modules, then by \cite[Section 5.3]{FHL}, there exists a non-degenerate bilinear form
\begin{align}
\langle~,~\rangle\colon\Wmod{\pQ+\lambda'}\times\Wmod{\pQl}\rightarrow\mathbb{C}
\end{align}
such that for any $a\in\Wmod{\pQ}$, $u\in\Wmod{\pQ+\lambda'}$, $v\in\Wmod{\pQl}$ and $n\in\mathbb{Z}$, we have
\begin{align}
\langle a_{(n)}u,v\rangle=\langle u,a_{(n)}^\dagger v\rangle.
\end{align}
Namely, $\langle~,~\rangle$ is a {\it non-degenerate $\Wmod{\pQ}$-invariant bilinear map}.
In particular, $\langle~,~\rangle$ is non-degenerate on $(\Wmod{\pQ+\lambda'})_{\Delta}\times(\Wmod{\pQl})_{\Delta}$ for each $\Delta\in\mathbb{Q}$, and we have $\langle (\Wmod{\pQ+\lambda'})_{\Delta}, (\Wmod{\pQl})_{\Delta'}\rangle=0$ if $\Delta\not=\Delta'$.

\begin{lemm}\label{lemm:1}
Let us fix an integer $p\geq 2$ and $\lambda\in\Lambda$.
If $\Wmod{\pQl}\simeq\Wmod{\pQ+\lambda'}^\ast$ as $\Wmod{\pQ}$-modules and $\Wmod{\pQ+\lambda'}$ is generated by its lowest conformal weight vectors (see Remark \ref{rmk:lowestconfweight}) as $\Wmod{\pQ}$-module, then for any nonzero $\Wmod{\pQ}$-submodule $S\subseteq\Wmod{\pQl}$, we have $S\cap(\Wmod{\pQl})_{\Delta_{\lambda}}\not=\{0\}$.
\end{lemm}
\begin{proof}
Let $S\subseteq\Wmod{\pQl}$ be a nonzero $\Wmod{\pQ}$-submodule in $\Wmod{\pQl}$. 
We take a nonzero vector $v=\sum v_{\Delta}\in S$, where $v_\Delta\in(\Wmod{\pQl})_\Delta$.
Let $\Delta'$ be the conformal weight such that $v_{\Delta'}\not=0$ and $\Delta'\geq\Delta$ for all $\Delta$ such that $v_\Delta\not=0$. Then by the non-degeneracy of $\langle~,~\rangle$ on $(\Wmod{\pQ+\lambda'})_{\Delta}\times(\Wmod{\pQl})_{\Delta}$, there exists an element 
\begin{align}
a=\sum(a_{i_1})_{(-n_1)}\cdots (a_{i_m})_{(-n_m)}(x\otimes|\lambda'\rangle)\in(\Wmod{\pQ+\lambda'})_{\Delta},
\end{align}
where $x\in\irr{\lz}$ and $\sum(a_{i_1})_{(-n_1)}\cdots (a_{i_m})_{(-n_m)}\in\mathcal{U}(\Wmod{\pQ})$, such that
\begin{align}\label{(615)}
0\not=\langle a,v_{\Delta'}\rangle=\langle x\otimes|\lambda'\rangle,\sum(a_{i_m})_{(-n_m)}^\dagger\cdots(a_{i_1})_{(-n_1)}^\dagger v_{\Delta'}\rangle.
\end{align}
Since $\Wmod{\pQl}=\bigoplus_{\Delta\geq\Delta_\lambda}(\Wmod{\pQl})_\Delta$ and $\Delta_{\lambda'}=\Delta_{\lambda}$, we obtain that 
\begin{align*}
\sum (a_{i_m})_{(-n_m)}^\dagger\cdots (a_{i_1})_{(-n_1)}^\dagger v=\sum(a_{i_m})_{(-n_m)}^\dagger\cdots(a_{i_1})_{(-n_1)}^\dagger v_{\Delta'}\in L(\lz)\otimes|\lambda\rangle
\end{align*}
is nonzero. Thus, $S$ contains a nonzero element of $\irr{\lz}\otimes|\lambda\rangle=(\Wmod{\pQl})_{\Delta_\lambda}$. 
\end{proof}

\subsection{Proof of Theorem \ref{mthm1} for $\lp=0$}\label{sect:generalequivsimple012}
We introduce a generalized vertex operator algebra 
\begin{align}\label{W(p)P}
\Wmod{\pP}:=\bigcap_{i=1}^l\ker S_{i,0}|_{\Vmod{\pP}}=\bigoplus_{\lz\in\Lz}\Wmod{\sqrt{p}(Q-\lz)}\simeq H^0(G\times_B\Vmod{\pP}),
\end{align}
and for $\lp\in\Lp$, the $\Wmod{\pP}$-modules
\begin{align}
\Wmod{\pP+\lp}:=\bigoplus_{\lz\in\Lz}\Wmod{\sqrt{p}(Q-\lz)+\lp}\simeq H^0(G\times_B\Vmod{\pPl}).
\end{align}
In the same manner as Corollary \ref{cor:thm1.2}, we obtain the following:

\begin{thm}\label{thm11/11}
We have the $\Wmod{\pP}$-module isomorphism
\begin{align}\label{WPversionofcor:thm1.2}
\Wmod{\pP+\lp}\simeq\Wmod{\pP+\lp'}^\ast.
\end{align}
\end{thm}

\begin{cor}\label{cor:1}
For $p\geq h-1$, $\Wmod{\pP}$ is simple.
\end{cor}
\begin{proof}
The condition $p\geq h-1$ is equivalent to $(\rho,\theta)\leq p$, namely, $(\sqrt{p}\lp+\rho,\theta)\leq p$ for the case $\lp=0$.
By Theorem \ref{thm11/11}, $\Wmod{\pP}$ is self-dual, and thus, has a non-degenerate $\Wmod{\pP}$-invariant bilinear form
$\Wmod{\pP}\times\Wmod{\pP}\rightarrow\mathbb{C}$.
Since $\Wmod{\pP}$ is generated by the vacuum vector $|0\rangle$, which is the unique eigenvector of the (lowest) conformal weight $0$ in $\Wmod{\pP}$, in the same manner as Lemma \ref{lemm:1}, $\Wmod{\pP}$ is simple.
\end{proof}

From now on, $\kappa$ and $\kappa'$ mean complex numbers such that $\kappa\not=-h$, $\kappa'\not=-h$, and $\kappa'=\frac{1}{\kappa+h}-h$.
Denote by $\uniwalg{\kappa}$ and $\irrwalg{\kappa}$ the {\it affine $W$-algebra} \cite{FF} of level $\kappa$ and the unique simple quotient, respectively.
Then we have the {\it Feigin-Frenkel duality} \cite{FF1,FF2,ACL}
\begin{align}\label{FFdualitykappa}
\uniwalg{\kappa}\simeq\uniwalg{\kappa'}.
\end{align}
From now on, we use the letters $k=p-h$ and $k'=\frac{1}{p}-h$.

\begin{thm}\label{mostgeneral}
For $p\geq h-1$ and $\lp=0$, Theorem \ref{mthm1} is true and $\wmod{\pal}$ is a simple $\irrwalg{k}$-module for any  $\alpha\in P_+\cap Q$.
%Theorem \ref{mthm2} are true . 
\end{thm}
\begin{proof}
By Corollary \ref{cor:1}, $\Wmod{\pP}$ is simple.
By applying \cite[Theorem 3.2]{McR} to $\Wmod{\pP}$, for any $\alpha\in P_+\cap Q$ and $\lz\in\Lz$, we obtain that $\wmod{-\sqrt{p}(\alpha+\lz)}$ is simple as $\wmod{0}$-module.
In particular, $\wmod{0}$ is simple.
On the other hand, by \cite[Theorem 1.2]{S}, we have $\wmod{0}\simeq\uniwalg{k}$ for $p\geq h-1$.
Thus, 
%Theorem \ref{mthm2} is proved 
for the case where $\lp=0$,
$\wmod{\pal}$ is simple as $\irrwalg{k}$-module.
Finally, by applying \cite[Proposition 2.26]{McR} to $\Wmod{\pP}$, we obtain that $\Wmod{\sqrt{p}(Q-\lz)}$ is simple as $\Wmod{\pQ}$-module for each $\lz\in\Lz$.
\end{proof}

By combining \eqref{FFdualitykappa} with the simplicity of $\uniwalg{k}\simeq\wmod{0}$ in Theorem \ref{mostgeneral}, for $p\geq h-1$, we obtain that
\begin{align}\label{FFdualityk}
\irrwalg{k}\simeq\uniwalg{k}\simeq\uniwalg{k'}\simeq\irrwalg{k'}.
\end{align}

\subsection{Proof of Theorem \ref{mthm2}}
Let us recall the notations in \cite[Chapter 6]{Kac2}.
In this paper, we consider the affine Lie algebras of type $A^{(1)}_l$, $D^{(1)}_l$, $E^{(1)}_6$, $E^{(1)}_7$, $E^{(1)}_8$.
For an affine Lie algebra $\hat{\mathfrak{g}}$ corresponding to the finite-dimensional simply-laced simple Lie algebra $\mathfrak{g}$, the Cartan subalgebra $\hat{\mathfrak{h}}$ and its dual $\hat{\mathfrak{h}}^\ast$ are decomposed as $\hat{\mathfrak{h}}=\mathfrak{h}\oplus(\mathbb{C}K+\mathbb{C}d)$ and $\hat{\mathfrak{h}}^\ast=\mathfrak{h}^\ast\oplus(\mathbb{C}\delta+\mathbb{C}\Lambda_0)$, respectively (there would be no risk of confusion with $\Lambda_0\subseteq P_+$ in Section \ref{sect:oursetting}).
For $\mu\in\hat{\mathfrak{h}}^\ast$, we use the letter $\bar\mu\in\mathfrak{h}^\ast$ for the classical part of $\mu$.
 By this notation, we have the decomposition
\begin{align}
\mu=\bar\mu+\langle\mu,K\rangle\Lambda_0+(\mu,\Lambda_0)\delta.
\end{align} 
The affine Weyl group $\hat{W}$ of $\hat{\mathfrak{g}}$ is given by $\hat{W}=W\ltimes Q$ (see \cite[Proposition 6.5]{Kac2}).
Here, the action 
\begin{align}
\hat{W}=W\ltimes Q\rightarrow\operatorname{GL}(\hat{\mathfrak{h}}^\ast),~(\sigma,\beta)\mapsto\sigma t_\beta
\end{align}
is defined by the following: for $\mu\in\hat{\mathfrak{h}}^\ast$, we have
\begin{align}
\sigma t_\beta(\mu)=\sigma(\bar\mu+\langle\mu,K\rangle\beta)+\langle\mu,K\rangle\Lambda_0+((\mu,\Lambda_0-\beta)-\frac{1}{2}|\beta|^2\langle\mu,K\rangle)\delta.
\end{align}
We also define the $\circ$ action of $\hat{W}$ on $\hat{\mathfrak{h}}^\ast$ by
\begin{align}
\sigma t_\beta\circ\mu=\sigma t_\beta(\mu+\rho+h\Lambda_0)-(\rho+h\Lambda_0).
\end{align}
By \cite[Proposition 6.3]{Kac2}, the set of positive real roots of $\hat{\mathfrak{g}}$ is given by
\begin{align}
\hat{\Delta}^+_{\operatorname{re}}=\{\bar\gamma+n\delta~|~\bar\gamma\in\Delta^\pm, n>0\}\cup\Delta^+,
\end{align}
where $\Delta^+$ and $\Delta^-$ are sets of positive and negative roots of $\mathfrak{g}$, respectively.
Set
\begin{align}
\hat{\mathcal{C}}^+=\{\mu\in\hat{\mathfrak{h}}^\ast~|~(\gamma,\mu+\rho+h\Lambda_0)\geq 0\text{ for any }\gamma\in\hat{\Delta}^+_{\operatorname{re}}\}.
\end{align}
By calculation, we obtain the following lemma.

\begin{lemm}
For $\lambda\in\Lambda$, $\sigma\in W$ and $\beta\in Q$, we have
\begin{align}
\sigma t_\beta\circ(-p(\alpha+\lz+\rho)+\sqrt{p}\lp+k\Lambda_0)\in\hat{\mathcal{C}}^+
\end{align}
if and only if
\begin{align}\label{8/20,10}
0\leq(\sigma^{-1}(\bar\gamma),p(\beta-(\alpha+\lz+\rho))+\sqrt{p}\lp+\rho)\leq p~\text{for any $\bar\gamma\in\Delta^+$.}
\end{align}
\end{lemm}

When $\lambda\in\Lambda$ satisfies $(\sqrt{p}\lp+\rho,\theta)\leq p$, a pair $(\sigma,\beta)$ satisfying \eqref{8/20,10} is given by the following.

\begin{lemm}\label{lemm,8/20,5}
Let $\alpha\in P_+\cap Q$ and $\lambda\in\Lambda$ such that $(\sqrt{p}\lp+\rho,\theta)\leq p$.
Then there exist $\omega_{\lz}\in\Lz$ and $\sigma_{\lz}\in W$ such that
$\beta_{\lz}=\alpha+\lz+\rho-\omega_{\lz}\in Q$ and 
\begin{align}\label{8/20,12.5}
\sigma_{\lz}^{-1}(\Delta^+)\cap\Delta^\pm=\{\pm\bar\gamma\in\Delta^\pm~|~(\bar\gamma,\omega_{\lz})=\frac{1\mp1}{2}\}.
\end{align}
Moreover, the pair $(\sigma_{\lz}, \beta_{\lz})$ satisfies \eqref{8/20,10} for $\lambda$.
\end{lemm}
\begin{proof}
Let us fix $\lambda\in\Lambda$ such that $(\sqrt{p}\lp+\rho,\theta)\leq p$.
First we prove the assertion in the case of $(\sqrt{p}\lp+\rho,\theta)< p$.
By \cite[Lemma 2.10]{KT}, there exist $\sigma\in W$ and $\beta=\alpha+\lz+\rho-\sum_{i=1}^ln_i\omega_i\in Q$ satisfying \eqref{8/20,10} for $\lambda$ (where $n_i\in\mathbb{Z}$).
Set
\begin{align}
I_\pm=\{i~|~\pm\alpha_i\in\sigma^{-1}(\Delta^+)\cap\Delta^\pm\}.
\end{align}
Clearly, we have $I:=\{1,\dots,l\}=I_+\sqcup I_-$.
For $x\in\Delta$, if there exists, we write $\bar\gamma_x$ for the positive root such that $\sigma^{-1}(\bar\gamma_x)=x$.
For each $i\in I_+$, by considering the case $\bar\gamma=\bar\gamma_{\alpha_i}$ in \eqref{8/20,10}, we have
\begin{align}\label{8/20,20}
0\leq (\alpha_i, -p\sum_{j\in I}n_j\omega_j+\sqrt{p}\lp+\rho)=-pn_i+(\sqrt{p}\lp+\rho,\alpha_i)\leq p.
\end{align}
By \eqref{8/20,20} and the assumption $(\sqrt{p}\lp+\rho,\theta)< p$, we have $n_i=0$ for any $i\in I_+$.
Similarly, for each $i\in I_-$, by considering the case $\bar\gamma=\bar\gamma_{-\alpha_i}$ in \eqref{8/20,10}, we have $n_i=1$ for any $i\in I_-$.
Thus, we obtain that
\begin{align}
\beta=\alpha+\lz+\rho-\sum_{i\in I_-}\omega_i.
\end{align}

Let us note that one of $\bar\gamma_{\theta}$ or $\bar\gamma_{-\theta}$ always exists.
If there exists $\bar\gamma_{\theta}$, then by considering the case $\bar\gamma=\bar\gamma_{\theta}$ in \eqref{8/20,10} with the assumption $(\sqrt{p}\lp+\rho,\theta)< p$, we have $(\theta,\sum_{i\in I_-}\omega_i)=0$. 
Since $\theta\geq\alpha_i$ for any $i\in I$, we obtain that $I_{-}=\phi$.
Thus, we have $\beta=\alpha+\lz+\rho$ (namely, $\omega_{\lz}=0$) and $\sigma_{\lz}:=\sigma=\operatorname{id}$, and they satisfy \eqref{8/20,12.5}.

On the other hand, if there exists $\bar\gamma_{-\theta}$, similarly we obtain that $I_{-}=\{\omega_i\}$ for some $i\in I$ such that $(\omega_i,\theta)=1$.
In particular, $\omega_i\in\Lz$.
Thus, for $\omega_{\lz}:=\omega_i$ and $\sigma_{\lz}:=\sigma$, it is easy to check that they satisfy  \eqref{8/20,12.5} by \eqref{8/20,10}.

Finally, let us extend the results above to the case of $(\sqrt{p}\lp+\rho,\theta)\leq p$.
Since $\omega_{\lz}$ and $\sigma_{\lz}$ do not depend on the choice of $\lp$ and they satisfy \eqref{8/20,12.5}, by applying \eqref{8/20,12.5} to \eqref{8/20,10}, we obtain the assertion.
\end{proof}

\begin{rmk}
It is easy to check that $\omega_{\lz}\in\Lz$ and $\sigma_{\lz}\in W$ in Lemma \ref{lemm,8/20,5} satisfy \eqref{8/20,10} if and only if $(\sqrt{p}\lp+\rho,\theta)\leq p$.
\end{rmk}

Let us recall notations in \cite{Ar}.
For an affine Lie algebra $\hat{\mathfrak{g}}$ and $\mu\in\hat{\mathfrak{h}}^\ast$, denote by $\afirr{\mu}$ and $\afverma{\mu}$ the corresponding irreducible highest weight module and Verma module of $\hat{\mathfrak{g}}$ with highest weight $\mu\in\hat{\mathfrak{h}}^\ast$. 
On the other hand, for $\bar\mu\in\mathfrak{h}^\ast$, 
let $\gamma_{\bar\mu}\colon Z(\mathfrak{g})\rightarrow\mathbb{C}$ be the evaluation at the Verma module $\verma{\bar\mu}$ of $\mathfrak{g}$ with highest weight $\bar\mu$, where $Z(\mathfrak{g})$ is the center of $\mathcal{U}(\mathfrak{g})$. We write $\bm{L}_{\kappa}(\gamma_{\bar\mu})$ and $\bm{M}_\kappa(\gamma_{\bar\mu})$ for the corresponding irreducible highest weight module and Verma module of $\uniwalg{\kappa}$, respectively. 
Then $\bm{L}_{\kappa}(\gamma_{\bar\mu})$ is the simple quotient of $\bm{M}_{\kappa}(\gamma_{\bar\mu})$ (see \cite[Section 5.3]{Ar}).
When $\kappa=k$ or $k'$, we will denote $\bm{L}_{\kappa}(\gamma_{\bar\mu})$ and $\bm{M}_{\kappa}(\gamma_{\bar\mu})$ simply by $\bm{L}(\gamma_{\bar\mu})$ and $\bm{M}(\gamma_{\bar\mu})$, respectively.
%Then $\wirr{\gamma_{\bar\mu}}$ is the simple quotient of $\wverma{\gamma_{\bar\mu}}$
By \cite[Lemma 4.2]{ArF}, we obtain the following lemma.
\begin{lemm}\label{ArF4.2}
Under the identification \eqref{FFdualitykappa}, for $\nu,\mu\in P$, we have 
\begin{align}\label{alignArF4.2}
\bm{L}_\kappa(\gamma_{\nu-(\kappa+h)(\mu+\rho)})\simeq\bm{L}_{\kappa'}(\gamma_{\mu-(\kappa'+h)(\nu+\rho)}).
\end{align}
\end{lemm}

%
%From now on, we identify 
%$\wirr{\gamma_{-p(\alpha+\lambda_0+\rho)+\sqrt{p}\lambda_p}}$ 
%with
%$\wirr{\gamma_{\alpha+\lambda_0-\frac{1}{p}(\sqrt{p}\lambda_p+\rho)}}$.
%We also identify 
%$\wirr{\gamma_{-p(\alpha+\lambda_0+\rho)+\sqrt{p}\lambda_p}}_{\operatorname{top}}\simeq\mathbb{C}_{\gamma_{-p(\alpha+\lambda_0+\rho)+\sqrt{p}\lambda_p}}$
%with
%$\wirr{\gamma_{\alpha+\lambda_0-\frac{1}{p}(\sqrt{p}\lambda_p+\rho)}}_{\operatorname{top}}\simeq\mathbb{C}_{\gamma_{\alpha+\lambda_0-\frac{1}{p}(\sqrt{p}\lambda_p+\rho)}}$.

For a $\uniwalg{\kappa}$-module $M$, we write $[M]$ for the element corresponding to $M$ in the Grothendieck group of the category of $\uniwalg{\kappa}$-modules.
By combining Lemma \ref{lemm,8/20,5} with the main results of \cite{Ar} and \cite{KT}, we obtain the following lemma.

\begin{lemm}\label{thmcharacter}
For $\alpha\in P_+\cap Q$ and $\lambda\in\Lambda$ such that $(\sqrt{p}\lp+\rho,\theta)\leq p$, we have
\begin{align}\label{Slackcomment}
[\wirr{\gamma_{-p(\alpha+\lz+\rho)+\sqrt{p}\lp}}]=\sum_{y\geq y_{\alpha,\lz}}a_{y,y_{\alpha,\lz}}[\wverma{\gamma_{\overline{y\circ\mu_\lambda}}}],
\end{align}
where 
$\geq$ in the sum is the Bruhat order, $\sigma_{\lz}$ and $\omega_{\lz}$ are defined in Lemma \ref{lemm,8/20,5}, 
\begin{align}
&y_{\alpha,\lz}=t_{\omega_{\lz}-(\alpha+\lz+\rho)}\sigma_{\lz}^{-1}\in\hat{W},\label{7/28,6}\\
&a_{y,{y_{\alpha,\lz}}}=(-1)^{l(y)-l(y_{\alpha,\lz})}Q_{y,y_{\alpha,\lz}}(1),\label{7/28,7}\\
&\mu_\lambda=\sigma_{\lz}(-p\omega_{\lz}+\sqrt{p}\lp+\rho)-\rho+k\Lz\in\hat{\mathcal{C}}^+,\label{7/28,8}
\end{align}
and $Q_{y',y}(z)$ is the inverse Kazhdan-Lusztig polynomial. 
In particular, we have 
\begin{align}\label{245}
\operatorname{tr}_{\wirr{\gamma_{-p(\alpha+\lz+\rho)+\sqrt{p}\lp}}}q^{L_0-\frac{c}{24}}=\sum_{y\geq y_{\alpha,\lz}}a_{y,y_{\alpha,\lp}}\frac{q^{\frac{1}{2p}|\overline{y\circ\mu_\lambda}+\rho|^2}}{\eta(q)^l}.
\end{align}
\end{lemm}

%\begin{lemm}\label{chandtopequal2}
%Let $M$ be a $\irrwalg{k}$-module and $N$ be a $\irrwalg{k}$-submodule of $M$.
%If
%\begin{align}
%&M/N
%\simeq\wirr{\gamma_{-p(\alpha+\lambda_0+\rho)+\sqrt{p}\lambda_p}},\label{topequal}\\
%%\end{align}
%%and
%%\begin{align}
%&\operatorname{tr}_{M}q^{L_0-\frac{c}{24}}
%=\operatorname{tr}_{\wirr{\gamma_{-p(\alpha+\lambda_0+\rho)+\sqrt{p}\lambda_p}}}q^{L_0-\frac{c}{24}},\label{charaequal}
%\end{align}
%then $M\simeq\wirr{\gamma_{-p(\alpha+\lambda_0+\rho)+\sqrt{p}\lambda_p}}$.
%\end{lemm}
%\begin{proof}
%By the assumption and Remark \ref{rmkinequality}, we have
%\begin{align}
%\operatorname{tr}_{M}q^{L_0-\frac{c}{24}}
%\geq\operatorname{tr}_{M/N}q^{L_0-\frac{c}{24}}
%=\operatorname{tr}_{\wirr{\gamma_{-p(\alpha+\lambda_0+\rho)+\sqrt{p}\lambda_p}}}q^{L_0-\frac{c}{24}}
%=\operatorname{tr}_{M}q^{L_0-\frac{c}{24}}.
%\end{align}
%Thus, by Remark \ref{rmkinequality}, $N=0$ and the claim is proved.
%\end{proof}

For a vertex operator algebra $V$ and a $V$-module $M$, let $M_{\operatorname{top}}$ be the sum of nonzero $M_{\Delta}$ such that $M_{\Delta-n}=0$ for all $n>0$.
Then $M_{\operatorname{top}}$ has the $\operatorname{Zhu}(V)$-module structure, where $\operatorname{Zhu}(V)$ is the {\it Zhu algebra} \cite{Z} of $V$ (see \cite[Section 3.12]{Ar}).
For $\mu\in P$, we have $\wirr{\gamma_{\mu}}_{\operatorname{top}}\simeq\mathbb{C}_{\gamma_{\mu}}$, where $\operatorname{Zhu}(\irrwalg{k})\simeq Z(\mathfrak{g})$ acts on $\mathbb{C}_{\gamma_{\mu}}$ by $\gamma_{\mu}$.

\begin{lemm}\label{newlemm.2022.4.23}
For $\alpha\in P_+\cap Q$ and $\lambda\in\Lambda$, we have 
\begin{align}\label{newlemm.2022.4.23align}
\wmod{\pal}_{\operatorname{top}}\simeq\mathbb{C}_{\gamma_{-p(\alpha+\lambda_0+\rho)+\sqrt{p}\lambda_p}}
\end{align}
as $\operatorname{Zhu}(\irrwalg{k})$-modules.
In particular, for $p\geq h-1$, we have 
\begin{align}\label{newlemm.2022.4.23align2}
\wmod{-\sqrt{p}(\alpha+\lambda_0)}
\simeq\wirr{\gamma_{-p(\alpha+\lambda_0+\rho)}}.
\end{align}
\end{lemm}

\begin{proof}
By Remark \ref{rmk2022.4.24} and the embedding $\uniwalg{\kappa}\hookrightarrow\wfock{0}$, for $\nu,\mu\in P$, we have
\begin{align}\label{align1,2022.4.25}
\wfock{\frac{\nu}{\sqrt{\kappa+h}}-\sqrt{\kappa+h}\mu}_{\operatorname{top}}\simeq(\pi_{\kappa+h,\nu-(\kappa+h)\mu})_{\operatorname{top}}\simeq\mathbb{C}_{\gamma_{\kappa+h,\nu-(\kappa+h)(\mu+\rho)}}
\end{align}
as $\operatorname{Zhu}(\irrwalg{\kappa})$-modules, where 
the second isomorphism follows from the proof of \cite[Proposition 6.2]{ACL} and the footnote in \cite[Section 6]{ACL}.
In particular, when $\kappa=k$, $\nu=\sqrt{p}\lp$, and $\mu=\alpha+\lz$, we have
\begin{align}\label{align1.5,2022.4.25}
\wfock{\pal}_{\operatorname{top}}\simeq\mathbb{C}_{\gamma_{-p(\alpha+\lambda_0+\rho)+\sqrt{p}\lambda_p}}.
\end{align}
By Definition \ref{defofwmodpal} of $\wmod{\pal}$ and the facts that 
\begin{align}
\wfock{\pal}=\bigoplus_{\Delta\in\Delta_{\pal}+\mathbb{Z}_{\geq 0}}\wfock{\pal}_\Delta
\end{align}
and $\wfock{\pal}_{\Delta_{\pal}}=\mathbb{C}|\pal\rangle$, we have
\begin{align}\label{align2,2022.4.25}
\wmod{\pal}_{\operatorname{top}}=\mathbb{C}|\pal\rangle=\wfock{\pal}_{\operatorname{top}}.
\end{align}
Thus, by \eqref{align1.5,2022.4.25} and \eqref{align2,2022.4.25}, we obtain \eqref{newlemm.2022.4.23align}.
The last assertion follows from \eqref{newlemm.2022.4.23align}, \cite[Theorem 5.3.1]{Ar}, and Theorem \ref{mostgeneral}.
\end{proof}

%\begin{rmk}
%When $p=h-1$, we have $(\sqrt{p}\lp+\rho,\theta)\leq p$ if and only if $\lp=0$, and thus Theorem \ref{mthm1} and Theorem \ref{mthm2} follow from Theorem  \ref{mostgeneral} and \eqref{newlemm.2022.4.23align2}, respectively. %for the proofs of Theorem \ref{mthm1} and Theorem \ref{mthm2}, 
%Hence it is enough to consider the cases where $p\geq h$ from now on.
%\end{rmk}
%

%\begin{lemm}\label{chandtopequal2}
%Let $M$ be a $\uniwalg{\kappa}$-module and $N$ be a $\uniwalg{\kappa}$-submodule of $M$.
%If
%\begin{align}
%&M/N\simeq\bm{L}_\kappa(\gamma_{\nu-(\kappa+h)(\mu+\rho)}),\label{topequal}\\
%%\end{align}
%%and
%%\begin{align}
%&\operatorname{tr}_{M}q^{L_0-\frac{c}{24}}
%=\operatorname{tr}_{\bm{L}_\kappa(\gamma_{\nu-(\kappa+h)(\mu+\rho)})}q^{L_0-\frac{c}{24}},\label{charaequal}
%\end{align}
%then $M\simeq\bm{L}_\kappa(\gamma_{\nu-(\kappa+h)(\mu+\rho)})$.
%\end{lemm}
%\begin{proof}
%By the assumption and Remark \ref{rmkinequality}, we have
%\begin{align}
%\operatorname{tr}_{M}q^{L_0-\frac{c}{24}}
%\geq\operatorname{tr}_{M/N}q^{L_0-\frac{c}{24}}
%=\operatorname{tr}_{\bm{L}_\kappa(\gamma_{\nu-(\kappa+h)(\mu+\rho)})}q^{L_0-\frac{c}{24}}
%=\operatorname{tr}_{M}q^{L_0-\frac{c}{24}}.
%\end{align}
%Thus, by Remark \ref{rmkinequality}, $N=0$ and the claim is proved.
%\end{proof}

Let us recall the following result in \cite{S}.

\begin{thm}[{\cite{ArF,S}}]\label{character}
For any $\alpha\in P_+\cap Q$ and $\lambda\in\Lambda$ such that $(\sqrt{p}\lp+\rho,\theta)\leq p$, we have the character formula
\begin{align}\label{(1984)}
\operatorname{tr}_{\wmod{\pal}}q^{L_0-\frac{c}{24}}=\sum_{\sigma\in W}(-1)^{l(\sigma)}\frac{q^{\frac{1}{2}|\sqrt{p}\sigma(\alpha+\lz+\rho)-\lp-\frac{1}{\sqrt{p}}\rho|^2}}{\eta(q)^l},
\end{align}
where $\eta(q)$ is the Dedekind eta function.
Furthermore, \eqref{(1984)} coincides with the characters of the $\irrwalg{k}$-modules $\warf{p}{\sqrt{p}\lp}{\alpha+\lz}$ and $\warf{\frac{1}{p}}{\alpha+\lz}{\sqrt{p}\lp}$
defined in \cite{ArF},
where $\warf{\kappa+h}{\nu}{\mu}=H^0_{\mu}(\afweyl{\nu+\kappa\Lambda_0})$,
$\afweyl{\nu+\kappa\Lambda_0}$ is the Weyl module over $\hat{\mathfrak{g}}$ of level $\kappa$ induced from $\irr{\nu}$, and $H^0_{\mu}(?)$ is the functor in \cite[Section 2.1]{ArF}. 
\end{thm}
When $\mu=0$, the functor $H^0_{\mu}(?)$ is the $``+"$ reduction $H^0_{+}(?)$ in \cite[Section 6.4]{Ar}.
\begin{rmk}\label{1rmk2022.5.5}
By \cite[Section 4.3]{ArF}, for $\nu,\mu\in P_+$ and $\kappa\in\mathbb{C}$ such that $\kappa\not=-h$, we have
$\operatorname{tr}_{\warf{\kappa+h}{\nu}{\mu}}q^{L_0-\frac{c}{24}}=\operatorname{tr}_{\warf{\kappa'+h}{\mu}{\nu}}q^{L_0-\frac{c}{24}}$.
\end{rmk}

\begin{lemm}\label{zerocase}
For any $\alpha\in P_+\cap Q$ and $p\geq h$, we have
\begin{align}\label{Slackcomment2}
[\warf{\frac{1}{p}}{\alpha+\lz}{0}]=\sum_{\sigma\in W}(-1)^{l(\sigma)}[\wverma{\gamma_{\overline{y_{\sigma}\circ\mu_{-\sqrt{p}\lz}}}}],
\end{align}
where $y_\sigma=t_{\sigma(\omega_{\lz})-(\alpha+\lz+\rho)}\sigma\sigma_{\lz}^{-1}\in\hat{W}$.
\end{lemm}
\begin{proof}
As in \cite{ArF}, we have the contragredient BGG resolution $C^{\bullet}(\mu)$ of $\irr{\mu}$ and the induced reduction $\hat{C}^{\bullet}(\mu)$ of $\hat{\mathfrak{g}}$-module of level $\kappa'$ such that
\begin{align}
C^{j}(\mu)&=\bigoplus_{\substack{\sigma\in W\\l(\sigma)=j}}\verma{\sigma\circ\mu}^\ast,\label{BGGresol1}\\
\hat{C}^{j}(\mu)&=\bigoplus_{\substack{\sigma\in W\\l(\sigma)=j}}
\afverma{\sigma\circ\mu+\kappa'\Lz}^{\ast},\label{BGGresol2}
\end{align}
where $\verma{\xi}^\ast$ is the contragredient Verma module over $\mathfrak{g}$ with highest weight $\xi\in\mathfrak{h}^\ast$, and $\afverma{\xi+\kappa'\Lz}^{\ast}$ is the corresponding induced $\hat{\mathfrak{g}}$-module of level $\kappa'$. 
Let us consider the case where $\kappa=k$ and $\mu=\alpha+\lz$.
Because $p\geq h$, $\alpha+\lz+k'\Lz$ satisfies the condition \cite[(385)]{Ar}.
By \cite[Theorem 9.1.3]{Ar}, the $``+"$ reduction $H^0_{+}(?)$ is the exact functor from $\mathcal{O}_{k'}$ (see \cite[Section 6.1]{Ar}) to the full subcategory of the category of graded $\irrwalg{k'}$-modules consisting of admissible $\irrwalg{k'}$-modules. Thus, by sending the induced BGG resolution $\hat{C}^{\bullet}(\alpha+\lz)$ by $H^0_{+}(?)$, we obtain the resolution $H^0_{+}(\hat{C}^{\bullet}(\alpha+\lz))$ of $\warf{\frac{1}{p}}{\alpha+\lz}{0}=H^0_+(\afweyl{\alpha+\lz+k'\Lz})$.
Since
$[H^0_+(\afverma{\sigma\circ(\alpha+\lz)+k'\Lz)^{\ast}}]=[\wverma{\gamma_{\overline{y_{\sigma}\circ\mu_{-\sqrt{p}\lz}}}}]$, we obtain the assertion.
\end{proof}

From now on, for $q$-series $A(q), B(q)\in q^{x}\mathbb{Z}[[q]]$, $x\in\mathbb{Q}$, the notion $A(q)\geq B(q)$ means that $A(q)-B(q)\in q^x\mathbb{Z}_{\geq 0}[[q]]$.

\begin{rmk}\label{rmkinequality}
For a vertex operator algebra $V$, a $V$-module $M$ and a $V$-submodule $N$ of $M$, we have $\operatorname{tr}_{M}q^{L_0-\frac{c}{24}}\geq\operatorname{tr}_{N}q^{L_0-\frac{c}{24}}$ and $\operatorname{tr}_{M}q^{L_0-\frac{c}{24}}\geq\operatorname{tr}_{M/N}q^{L_0-\frac{c}{24}}$, and the equal signs hold if and only if $M=N$ and $N=0$, respectively.
\end{rmk}

\begin{lemm}\label{2022.5.5}
For $\nu,\mu\in P_+$ and $\kappa\in\mathbb{C}$ such that $\kappa\not=-h$, if
\begin{align}\label{1align2022.5.5}
\operatorname{tr}_{\warf{\kappa+h}{\nu}{\mu}}q^{L_0-\frac{c}{24}}=\operatorname{tr}_{\wirr{\gamma_{\nu-(\kappa+h)(\mu+\rho)}}}q^{L_0-\frac{c}{24}},
\end{align}
then $\warf{\kappa+h}{\nu}{\mu}\simeq\bm{L}_{\kappa}(\gamma_{\nu-(\kappa+h)(\mu+\rho)})\simeq\bm{L}_{\kappa'}(\gamma_{\mu-(\kappa'+h)(\nu+\rho)})\simeq\warf{\kappa'+h}{\mu}{\nu}$ as $\uniwalg{\kappa}$-modules.
\end{lemm}
\begin{proof}
By Lemma \ref{ArF4.2} and Remark \ref{1rmk2022.5.5}, it is enough to show that $T^{\kappa+h}_{\nu,\mu}\simeq\bm{L}_{\kappa}(\gamma_{\nu-(\kappa+h)(\mu+\rho)})$.
By combining the injective homomorphism $\hat{V}(\nu+\kappa\Lambda_0)\hookrightarrow\hat{M}(\nu+\kappa\Lambda_0)^\ast$ in the BGG resolution \eqref{BGGresol2} with
the canonical homomorphism $\hat{M}(\nu+\kappa\Lambda_0)^\ast\rightarrow\hat{W}(\nu+\kappa\Lambda_0)$, where $\hat{W}(\nu+\kappa\Lambda_0)$ is the Wakimoto module of highest weight $\nu$ and level $\kappa$ (see \cite[Section 3]{ArF}), we obtain the $\hat{\mathfrak{g}}$-module homomorphism
\begin{align}
\phi\colon\hat{V}(\nu+\kappa\Lambda_0)\rightarrow\hat{W}(\nu+\kappa\Lambda_0),
\end{align}
which is isomorphic on the eigenspaces with lowest conformal weights (see \cite[Proposition 3.4]{ArF}).
In particular, $\phi$ sends the highest weight vector $v_\nu$ of $\hat{V}(\nu+\kappa\Lambda_0)$ to the highest weight vector $v'_\nu$ of $\hat{W}(\nu+\kappa\Lambda_0)$.
Denote by $[v_\nu]\in H^0_{\mu}(\hat{V}(\nu+\kappa\Lambda_0))=\warf{\kappa+h}{\nu}{\mu}$ and $[v'_\nu]\in H^0_{\mu}(\hat{W}(\nu+\kappa\Lambda_0))$ the equivalence classes of $v_\nu\otimes1_\mu$ and $v'_\nu\otimes 1_\mu$, respectively (see the definition of $H^0_{\mu}(?)$ in \cite[Section 2.1]{ArF}).
Then, clearly we have $H^0_{\mu}(\phi)([v_\nu])=[v'_\nu]$.
On the other hand, by \cite[Lemma 3.2]{ArF} and Remark \ref{rmk2022.4.24}, we obtain the $\uniwalg{\kappa}$-module isomorphism
\begin{align}
\phi' \colon H^0_{\mu}(\hat{W}(\nu+\kappa\Lambda_0))
\simeq\pi_{\kappa+h,\nu-(\kappa+h)\mu}
\simeq\wfock{\frac{\nu}{\sqrt{\kappa+h}}-\sqrt{\kappa+h}\mu},
\end{align}
which sends $[v'_\nu]$ to the highest weight vector $|\frac{\nu}{\sqrt{\kappa+h}}-\sqrt{\kappa+h}\mu\rangle$ of $\wfock{\frac{\nu}{\sqrt{\kappa+h}}-\sqrt{\kappa+h}\mu}$.
Thus, the $\uniwalg{\kappa}$-module homomorphism $\phi''=\phi'\circ H^0_\mu(\phi)$ sends $[v_\nu]$ to $|\frac{\nu}{\sqrt{\kappa+h}}-\sqrt{\kappa+h}\mu\rangle$, and we have
%By constructions of $H^0_{\mu}(?)$ (see \cite[Section 2]{ArF}) and $\phi$, the image of the $\irrwalg{\kappa}$-module homomorphism $\phi''=\phi'\circ H^0_{\mu}(\phi)$ contains the highest weight vector $|\frac{\nu}{\sqrt{\kappa+h}}-\sqrt{\kappa+h}\mu\rangle$ of $\wfock{\frac{\nu}{\sqrt{\kappa+h}}-\sqrt{\kappa+h}\mu}$.
\begin{align}\label{1align2022.4.30}
\warf{\kappa+h}{\nu}{\mu}/\operatorname{ker}
%H^0_{\alpha+\lz}(\phi)
\phi''\simeq\operatorname{Im}
%H^0_{\alpha+\lz}(\phi)
\phi''\supseteq\mathcal{U}(\uniwalg{\kappa})|\frac{\nu}{\sqrt{\kappa+h}}-\sqrt{\kappa+h}\mu\rangle.
\end{align}
By combining \eqref{align1,2022.4.25} and \eqref{1align2022.4.30} with Remark \ref{rmkinequality}, we have
\begin{align}\label{2align2022.4.30}
\operatorname{tr}_{\warf{\kappa+h}{\nu}{\mu}}q^{L_0-\frac{c}{24}}
\geq\operatorname{tr}_{\warf{\kappa+h}{\nu}{\mu}/\operatorname{ker}\phi''}q^{L_0-\frac{c}{24}}
\geq\operatorname{tr}_{\wirr{\gamma_{\nu-(\kappa+h)(\mu+\rho)}}}q^{L_0-\frac{c}{24}}.
\end{align}
The assertion follows from the assumption \eqref{1align2022.5.5}, \eqref{2align2022.4.30}, and Remark \ref{rmkinequality}.
%Let us consider the case where $\kappa=k$, $\mu=\alpha+\lz$ and $\nu=\sqrt{p}\lp$.
%By combining \eqref{2align2022.4.30} and \eqref{align1.2022.5.1} with Remark \ref{rmkinequality}, we obtain that
%\begin{align}\label{align2.2022.5.1}
%\warf{p}{\sqrt{p}\lp}{\alpha+\lz}\simeq\wirr{\gamma_{\sqrt{p}\lp-p(\alpha+\lz+\rho)}}.
%\end{align}
%On the other hand, let us consider the case where $\kappa=k'$, $\mu=\sqrt{p}\lp$ and $\nu=\alpha+\lz$.
%In the sama manner as above, we obtain that
%\begin{align}\label{align3.2022.5.1}
%\warf{\frac{1}{p}}{\alpha+\lz}{\sqrt{p}\lp}\simeq\wirr{\gamma_{\alpha+\lz-\frac{1}{p}(\sqrt{p}\lp+\rho)}}.
%\end{align}
%Theorem \ref{mthm2} follows from \eqref{0align2022.4.30}, \eqref{alignArF4.2}, \eqref{align2.2022.5.1} and \eqref{align3.2022.5.1}.
\end{proof}

%\begin{cor}\label{cor2022.5.5}
%For $\alpha\in P_+\cap Q$, $p\geq h-1$, and $\lambda\in\Lambda$ such that $(\sqrt{p}\lp+\rho,\theta)\leq p$, $\warf{p}{\sqrt{p}\lp}{\alpha+\lz}\simeq\bm{L}_{k}(\gamma_{-p(\alpha+\lz+\rho)+\sqrt{p}\lp})\simeq\bm{L}_{k'}(\gamma_{\alpha+\lz-\frac{1}{p}(\sqrt{p}\lp+\rho)})\simeq\warf{\frac{1}{p}}{\alpha+\lz}{\sqrt{p}\lp}$ as $\irrwalg{k}$-modules.
%\end{cor}

\begin{rmk}
When $\kappa=k'$, $p\geq h$, $\nu=\alpha+\lz$, and $\mu=0$, Lemma \ref{2022.5.5} follows from \cite[Theorem 9.1.4]{Ar} and the exactness of $H^0_+(?)$ in the proof of Lemma \ref{zerocase}.
\end{rmk}

\begin{thm}\label{premthm2}
Theorem \ref{mthm2} is true.
%For $p\geq h-1$ and $\lambda\in\Lambda$ such that $(\sqrt{p}\lp+\rho,\theta)\leq p$, we have $\wmod{\pal}\simeq\wirr{\gamma_{-p(\alpha+\lz+\rho)+\sqrt{p}\lp}}$ as $\irrwalg{k}$-modules.
\end{thm}
\begin{proof}
If $\lp=0$, then the assertion 
%has been already proved in Lemma \ref{newlemm.2022.4.23}.
follows from Lemma \ref{newlemm.2022.4.23}, Theorem \ref{character}, and Lemma \ref{2022.5.5}.
In particular, when $p=h-1$, we have $(\sqrt{p}\lp+\rho,\theta)\leq p$ if and only if $\lp=0$, and thus the claim is proved. 

Let us consider the case $p\geq h$.
By combining Lemma \ref{thmcharacter} and Lemma \ref{zerocase} with Theorem \ref{mthm2} for the case of $\lp=0$, we have
\begin{align}\label{250}
\sum_{\sigma\in W}(-1)^{l(\sigma)}[\wverma{\gamma_{\overline{y_{\sigma}\circ\mu_{-\sqrt{p}\lz}}}}]=\sum_{y\geq y_{\alpha,\lz}}a_{y,y_{{\alpha,\lz}}}[\wverma{\gamma_{\overline{y\circ\mu_{-\sqrt{p}\lz}}}}].
\end{align}
Let us recall the fact that 
\begin{align}\label{7/28,4}
[\wverma{\gamma_{\overline{y\circ\mu_\lambda}}}]=[\wverma{\gamma_{\overline{y'\circ\mu_\lambda}}}]
\Leftrightarrow y'\in Wy
\end{align}
for $y,y'\in\hat{W}$. Set the equivalence relation on $\hat{W}$ by $y\sim y'$ when $y'\in Wy$. Then by \eqref{250}, for any $y\in\hat{W}/\sim$, we have
\begin{align}\label{251}
\sum_{\substack{y'\geq y_{\alpha,\lz}\\y'\sim y}}a_{y',y_{\alpha,\lz}}=
\begin{cases}
(-1)^{l(\sigma)}&y=y_{\sigma}\text{~for some $\sigma\in W$},\\
0&\text{otherwise}.
\end{cases}
\end{align}
By applying \eqref{251} to \eqref{245}, we have
\begin{align}\label{traceequation}
\operatorname{tr}_{\wirr{\gamma_{-p(\alpha+\lz+\rho)+\sqrt{p}\lp}}}q^{L_0-\frac{c}{24}}
=&\sum_{y\geq y_{\alpha,\lz}}a_{y,y_{\alpha,\lz}}\frac{q^{\frac{1}{2p}|\overline{y\circ\mu_\lambda}+\rho|^2}}{\eta(q)^l}\\
=&\sum_{\substack{y\in\hat{W}/\sim\\ y\geq y_{\alpha,\lz}}}(\sum_{\substack{y'\geq y_{\alpha,\lz}\\ y'\sim y}}a_{y',y_{\alpha,\lz}})\frac{q^{\frac{1}{2p}|\overline{y\circ\mu_\lambda}+\rho|^2}}{\eta(q)^l}\nonumber\\
=&\sum_{\sigma\in W}(-1)^{l(\sigma)}\frac{q^{\frac{1}{2p}|\overline{y_\sigma\circ\mu_\lambda}+\rho|^2}}{\eta(q)^l}\nonumber\\
=&\operatorname{tr}_{\wmod{\pal}}q^{L_0-\frac{c}{24}}.\nonumber
\end{align}

Let 
$
\wmod{\pal}'=\mathcal{U}(\irrwalg{k})|\pal\rangle
$
be the $\irrwalg{k}$-submodule of $\wmod{\pal}$.
Since $\wmod{\pal}'$ is a highest weight $\irrwalg{k}$-module and 
\begin{align}
\wmod{\pal}'_{\operatorname{top}}=\wmod{\pal}_{\operatorname{top}}\simeq\mathbb{C}_{\gamma_{-p(\alpha+\lambda_0+\rho)+\sqrt{p}\lambda_p}}
\end{align}
(see Lemma \ref{newlemm.2022.4.23}),
we have
\begin{align}\label{maxmoduledash}
\wmod{\pal}'/\wmod{\pal}'_{\operatorname{max}}\simeq\wirr{\gamma_{-p(\alpha+\lz+\rho)+\sqrt{p}\lp}},
\end{align}
where $\wmod{\pal}'_{\operatorname{max}}$ is the maximal proper $\irrwalg{k}$-submodule of $\wmod{\pal}'$.
By Remark \ref{rmkinequality} and \eqref{traceequation}, \eqref{maxmoduledash}, we have
\begin{align}\label{traceinequakity3}
\operatorname{tr}_{\wmod{\pal}}q^{L_0-\frac{c}{24}}
&\geq\operatorname{tr}_{\wmod{\pal}'}q^{L_0-\frac{c}{24}}\\
&\geq\operatorname{tr}_{\wirr{\gamma_{-p(\alpha+\lz+\rho)+\sqrt{p}\lp}}}q^{L_0-\frac{c}{24}}\nonumber\\
&=\operatorname{tr}_{\wmod{\pal}}q^{L_0-\frac{c}{24}}.\nonumber
\end{align}
%where the first and second inequality signs, and the last equal sign follow from Remark \ref{rmkinequality} and \eqref{traceequation}, respectively.
Hence, by Remark \ref{rmkinequality} and \eqref{maxmoduledash}, \eqref{traceinequakity3}, we have 
 %and Lemma \ref{chandtopequal2}, we have
\begin{align}\label{0align2022.4.30}
\wmod{\pal}=\wmod{\pal}'\simeq\wirr{\gamma_{-p(\alpha+\lz+\rho)+\sqrt{p}\lp}}.
\end{align}
%In particular, by \eqref{alignArF4.2}, Theorem \ref{character}, and \eqref{0align2022.4.30}, we have
%\begin{align}\label{align1.2022.5.1}
%&\operatorname{tr}_{\warf{p}{\sqrt{p}\lp}{\alpha+\lz}}q^{L_0-\frac{c}{24}}
%=\operatorname{tr}_{\wirr{\gamma_{-p(\alpha+\lz+\rho)+\sqrt{p}\lp}}}q^{L_0-\frac{c}{24}}\\
%=&\operatorname{tr}_{\wirr{\gamma_{\alpha+\lz-\frac{1}{p}(\sqrt{p}\lp+\rho)}}}q^{L_0-\frac{c}{24}}
%=\operatorname{tr}_{\warf{\frac{1}{p}}{\alpha+\lz}{\sqrt{p}\lp}}q^{L_0-\frac{c}{24}}.\nonumber
%\end{align}
The remaining assertion follows from Theorem \ref{character}, Lemma \ref{2022.5.5}, and \eqref{0align2022.4.30}.
\end{proof}

%Let us prove the remaining part of Theorem \ref{mthm2}.

%\begin{cor}
%For any $\alpha\in P_+\cap Q$ and $\lambda\in\Lambda$ such that $(\sqrt{p}\lp+\rho,\theta)\leq p$, we have $\wmod{\pal}\simeq\warf{p}{\sqrt{p}\lp}{\alpha+\lz}\simeq\warf{\frac{1}{p}}{\alpha+\lz}{\sqrt{p}\lp}$ as $\irrwalg{k}$-modules, and they are simple (this corollary is an extension of \cite[Theorem 2.2, Theorem 2.3]{ArF}). 
%\end{cor}
%

\subsection{Proof of Theorem \ref{mthm1}}\label{sect:generalequivsimple11}
For $\beta\in P_+$, denote by $v_{\beta}$ the highest weight vector of
$\irr{\beta}$.
For $\alpha\in P_+\cap Q$, $\lz\in\Lz$ and $A_{\alpha+\lz}\in\mathbb{C}v_{\alpha+\lz}\otimes\wmod{\pal}$, we use the notation
\begin{align}\label{eq1111}
x\otimes A_{\alpha+\lz}=\sum f_{i_1}\cdots f_{i_n}A_{\alpha+\lz},
\end{align}
where $x=\sum f_{i_1}\cdots f_{i_n}v_{\alpha+\lz}\in L(\alpha+\lz)$ and $f_{i_j}$ in the right hand side of \eqref{eq1111} are the screening operators \eqref{(21.5)}.

\begin{lemm}\label{simplegeneral}
Let us take $\nu_p\in\Lp$ such that $\Wmod{\pP+\nu_p}\simeq\Wmod{\pP+\nu'_p}^\ast$ as $\Wmod{\pP}$-modules.
Let us also assume that  for any $\beta\in P_+$, $\wmod{-\sqrt{p}\beta+\nu_p}$ and $\wmod{-\sqrt{p}\beta+\nu'_p}$ are generated by $|-\sqrt{p}\beta+\nu_p\rangle$ and $|-\sqrt{p}\beta+\nu'_p\rangle$ as $\wmod{0}$-modules, respectively. 
Then $\Wmod{\pP+\nu_p}$ is simple as $\Wmod{\pP}$-module.
Moreover, for any $\lambda\in\Lambda$ such that $\lp=\nu_p$ and $\Wmod{\pQl}\simeq\Wmod{\pQ+\lambda'}^\ast$ as $\Wmod{\pQ}$-modules, $\Wmod{\pQl}$ is the simple $\Wmod{\pQ}$-module.
\end{lemm}
\begin{proof}
Because $\wmod{-\sqrt{p}\beta+\nu_p}=\mathcal{U}(\wmod{0})|-\sqrt{p}\beta+\nu_p\rangle$ and 
\begin{align}\label{align:decomps}
f_{\i_1}\cdots f_{i_n}|-\sqrt{p}\beta+\nu_p\rangle&=(f_{i_1}\cdots f_{i_n}|-\sqrt{p}\beta\rangle)_{(\Delta_{-\sqrt{p}\beta}+\Delta_{\nu_p}-\Delta_{-\sqrt{p}\beta+\nu_p}-1)}|\nu_p\rangle
\end{align}
(see \eqref{fidiff}), $\Wmod{\pP+\nu_p}$ is generated by $|\nu_p\rangle$ as $\Wmod{\pP}$-module.
Similarly, $\Wmod{\pP+\nu'_p}$ is generated by $|\nu'_p\rangle$ as $\Wmod{\pP}$-module. By ($\Wmod{\pP}$-version of) Lemma \ref{lemm:1}, $\Wmod{\pP+\nu_p}$ and $\Wmod{\pP+\nu'_p}$ are simple $\Wmod{\pP}$-modules.

Let us take $\lambda\in\Lambda$ such that $\lp=\nu_p$, and let $S$ be a nonzero $\Wmod{\pQ}$-submodule in $\Wmod{\pQl}$. 
In the same manner as above, $\Wmod{\pQ+\lambda'}$ is generated by $L(\lz')\otimes|\lambda'\rangle$ as $\Wmod{\pQ}$-module. 
Then by Lemma \ref{lemm:1}, $S$ contains a nonzero element $y\otimes|\lambda\rangle\in L(\lz)\otimes|\lambda\rangle$. 
Since $\Wmod{\pP+\nu_p}$ is simple as $\Wmod{\pP}$-module, by \cite[Corollary 4.2]{DM}, 
for any $x\in L(\lz)$, there exist elements $\{a_{n,x}\in\Wmod{\pP}\}_{n\in\mathbb{Z}}$ such that
\begin{align}
x\otimes|\lambda\rangle=\sum_{n\in\mathbb{Z}}(a_{n,x})_{(n)}y\otimes|\lambda\rangle
\end{align}
However, since $x, y\in L(\lz)$, we can assume that $a_{n,x}\in\Vmod{\pQ}\cap\Wmod{\pP}=\Wmod{\pQ}$ for any $x\in\irr{\lz}$ and $n\in\mathbb{Z}$. Thus, we have 
\begin{align}
\irr{\lz}\otimes|\lambda\rangle\in\mathcal{U}(\Wmod{\pQ})(y\otimes|\lambda\rangle)\subseteq S
\end{align}
and the claim is proved.
\end{proof}

\begin{rmk}
\cite[Proposition 4.1, Corollary 4.2]{DM} are claims about not generalized vertex operator algebras, but vertex operator algebras. 
However, it is easily checked that their proofs hold for generalized vertex operator algebras as well.
\end{rmk}

\begin{thm}
Theorem \ref{mthm1} is true.
\end{thm} 
\begin{proof}
By Theorem \ref{mthm2}, Corollary \ref{cor:thm1.2} and Theorem \ref{thm11/11}, all $\lambda\in\Lambda$ such that $(\sqrt{p}\lp+\rho,\theta)\leq p$ satisfy the assumptions in Lemma \ref{simplegeneral}. Thus, the claim is proved.
\end{proof}

\end{document}